\documentclass[11pt,a4paper]{amsart}
\usepackage{mathrsfs}
\usepackage{syntonly}
\usepackage{amsmath}
\usepackage{amsthm}
\usepackage{amsfonts}
\usepackage{amssymb}
\usepackage{latexsym}
\usepackage{amscd,amssymb,amsopn,amsmath,amsthm,graphics,amsfonts,mathrsfs,accents,enumerate,verbatim,calc}
\usepackage[dvips]{graphicx}
\usepackage[colorlinks=true,linkcolor=red,citecolor=blue]{hyperref}
\usepackage[all]{xy}
\usepackage{enumitem}

\date{}
\allowdisplaybreaks[4] \footskip=15pt
\renewcommand{\uppercasenonmath}[1]{}

\numberwithin{equation}{section} \theoremstyle{plain}
\newtheorem{theorem}{Theorem}[section]
\newtheorem{corollary}[theorem]{Corollary}
\newtheorem{lemma}[theorem]{Lemma}
\newtheorem{proposition}[theorem]{Proposition}
\theoremstyle{definition}
\newtheorem{definition}[theorem]{Definition}
\newtheorem{example}[theorem]{Example}
\newtheorem{remark}[theorem]{Remark}

\newtheorem*{ack*}{ACKNOWLEDGEMENTS}

\newcommand{\oo}{\otimes}
\newcommand{\ds}{\oplus}

\newcommand{\pf}{\noindent\begin {proof}}
\newcommand{\epf}{\end{proof}}

\newcommand{\ra}{\rightarrow}

\newcommand{\Hom}{\mbox{\rm Hom}}
\newcommand{\Tor}{\mbox{\rm Tor}}
\newcommand{\im}{\mbox{\rm im}}

\newcommand{\ke}{\mbox{\rm ker}}

\newcommand{\DMod}{\Delta\text{-}\mathrm{Mod}}
\newcommand{\AMod}{A\text{-}\mathrm{Mod}}

\newcommand{\BMod}{B\text{-}\mathrm{Mod}}
\newcommand{\ModD}{\mathrm{Mod}\text{-}\Delta}

\newcommand{\ModR}{\mathrm{Mod}\text{-}R}
\newcommand{\RMod}{R\text{-}\mathrm{Mod}}

\begin{document}
\begin{center}
{
{\bf\large How to construct Gorenstein projective modules relative to complete duality pairs over Morita  rings}\\

\vspace{0.5cm}   Yajun Ma, Jiafeng L${\rm\ddot{u}}$, Huanhuan Li and
Jiangsheng Hu}\\
\end{center}

 $$\bf  Abstract$$
\leftskip0truemm \rightskip0truemm \noindent Let $\Delta =\left(\begin{smallmatrix}  A & {_AN_B}\\  {_BM_A} & B \\\end{smallmatrix}\right)$ be a Morita ring with $M\oo_{A}N=0=N\oo_{B}M$.
 We first study how to construct
(complete) duality pairs of $\Delta$-modules using (complete) duality pairs of $A$-modules and $B$-modules, generalizing the result of Mao (Comm. Algebra, 2020, 12: 5296--5310) about the duality pairs over a triangular matrix ring. Moreover, we construct Gorenstein projective modules relative to complete duality pairs of $\Delta$-modules. Finally, we give an application to Ding projective modules.
\leftskip10truemm \rightskip10truemm \noindent
\\[2mm]

\leftskip0truemm \rightskip0truemm
\section { \bf Introduction}
The notion of duality pairs was introduced by Holm and J{\o}gensen in \cite{HP}, which is very useful in relative homological algebra because duality pairs are very closely related to purity, existences of covers and envelopes, and to complete cotorsion pairs. Recently, in order to show how Gorenstein homological algebra can be done with respect to a
duality pair, the notion of a complete duality pair was introduced by Gillespie in \cite{Gillespie}. For a given complete duality pair $(\mathcal{L},\mathcal{A})$, he introduced the notion of Gorenstein $(\mathcal{L},\mathcal{A})$-projective modules (see \cite[Definition 4.2]{Gillespie} or Definition \ref{definition:4.1} below). Indeed, if $R$ is a commutative Noetherian ring of finite Krull dimension, then this definition, applied to the flat-injective duality pair, agrees with the usual definition of Gorenstein projective modules studied by Enochs and many other authors. See \cite{EJ} for the basic theory of Gorenstein projective modules.

One of interesting topics in Gorenstein homological algebra is to determine Gorenstein projective modules over the triangular matrix ring
$\Delta =\left(\begin{smallmatrix}  A & {_AN_B}\\  0 & B \\\end{smallmatrix}\right)$ under some conditions on the bimodule $_AN_B$. See \cite{LZ,XZ12,zh13,LZHZ} for instance. A natural extension of triangular matrix rings is the class of Morita rings. Recall that
Morita rings are $2\times2$ matrix rings associated to Morita contexts \cite{B,C}. For more recent results on Morita rings, we refer to \cite{GP,Gao,Gao1,YY}.  A particular case of interest is the Morita ring $\Delta =\left(\begin{smallmatrix}  A & {_{A}}N_{B}\\  {_{B}}M_{A} & B \\\end{smallmatrix}\right)$ with $M\oo_{A}N=0=N\oo_{B}M$.  The reason is that the natural functors $\mathrm{T}_{A}$ and $\mathrm{H}_{A}$ (resp., $\mathrm{T}_{B}$ and $\mathrm{H}_{B}$) defined by Green-Psaroudakis in \cite{GP} preserve the left half of duality pairs from the category of left $A$-modules (resp., left $B$-modules) to the category of left $\Delta$-modules (see Lemma \ref{lem2} and Theorem \ref{thm1}).
The main problem considered in this paper is:

\textbf{Problem.} Construct Gorenstein projective modules relative to complete duality pairs over the Morita ring $\Delta =\left(\begin{smallmatrix}  A & {_{A}}N_{B}\\  {_{B}}M_{A} & B \\\end{smallmatrix}\right)$ with $M\oo_{A}N=0=N\oo_{B}M$.

In order to solve this problem, we first give a method to construct duality pairs over the Morita ring $\Delta$. To this end, for any class $\mathcal{C}_{1}$ (resp., $\mathcal{C}_{2}$) of left (resp., right) $A$-modules and any class $\mathcal{D}_{1}$ (resp., $\mathcal{D}_{2}$) of left (resp., right) $B$-modules, we denote by $\mathfrak{B}^{\mathcal{C}_{1}}_{\mathcal{D}_{1}}$ the class of left $\Delta$-modules $(X,Y,f,g)$
such that $f:M\oo_{A}X\rightarrow Y$ and $g:N\oo_{B}Y\rightarrow X$ are monomorphisms, $X/\im(g)\in \mathcal{C}_{1}$ and $Y/\im(f)\in\mathcal{D}_{1},$ and by $\mathfrak{J}_{\mathcal{C}_{2},\mathcal{D}_{2}}$ the class of right $\Delta$-modules $(X,Y,f,g)$ such that $\tilde{f}:X\rightarrow \Hom_{B}(N,Y)$ and $\tilde{g}:Y\rightarrow \Hom_{A}(M,X)$ are epimorphisms, $\ke \tilde{f}\in \mathcal{C}_{2}$ and $\ke \tilde{g}\in \mathcal{D}_{2},$ where $\tilde{f}(x)(n)=f(x\oo n)$ for $f\in \Hom_{B}(X\oo_{A}N,Y), x\in X, n\in N$ and $\tilde{g}(y)(m)=g(y\oo m)$ for $g\in \Hom_{A}(Y\oo_{B}M,X),y\in Y, m\in M.$
  We have the following theorem which is contained in Theorem \ref{thm1} and Corollary \ref{complete} below.

\begin{theorem}\label{thm1'} Suppose that $\Delta =\left(\begin{smallmatrix}  A & {_{A}}N_{B}\\  {_{B}}M_{A} & B \\\end{smallmatrix}\right)$ is a Morita ring with $M\oo_{A}N=0=N\oo_{B}M$. Let $\mathcal{C}_{1}$ (resp., $\mathcal{C}_{2}$) be a class of left (resp., right) $A$-modules and $\mathcal{D}_{1}$ (resp., $\mathcal{D}_{2}$) be a class of left (resp., right) $B$-modules.
\begin{enumerate}
 \item $(\mathcal{C}_{1},\mathcal{C}_{2})$ and $(\mathcal{D}_{1},\mathcal{D}_{2})$ are duality pairs if and only if $(\mathfrak{B}^{\mathcal{C}_{1}}_{\mathcal{D}_{1}},
\mathfrak{J}_{{\mathcal{C}_{2}}, {\mathcal{D}_{2}}})$ is a duality pair.

 \item If $N_{B}$ and $M_{A}$ are finitely generated projective modules, then $(\mathcal{C}_{1},\mathcal{C}_{2})$ and $(\mathcal{D}_{1},\mathcal{D}_{2})$ are complete duality pairs if and only if $(\mathfrak{B}^{\mathcal{C}_{1}}_{\mathcal{D}_{1}},
\mathfrak{J}_{{\mathcal{C}_{2}},{\mathcal{D}_{2}}})$ is a complete duality pair.
\end{enumerate}

\end{theorem}
We note that Theorem \ref{thm1'}(1) generalizes the equivalence of (1) and (2) in \cite[Theorem 2.1]{mao} to a more general case, where
the author assumes the ring $\Delta$ to be a triangular matrix ring.

Assume that $\mathcal{L}$ (resp., $\mathcal{A}$) is a class of left (resp., right) $R$-modules over a ring $R$ such that $(\mathcal{L},\mathcal{A})$ is a complete duality pair of $R$-modules. Recall from \cite[Definition 4.2]{Gillespie} that a left $R$-module $X$ is called Gorenstein $(\mathcal{L},\mathcal{A})$-projective if there exists an exact complex of projective left $R$-modules
$$\mathbb{P}:\cdots\ra P^{-1}\xrightarrow{d^{-1}} P^{0}\xrightarrow{d^{0}} P^{1}\ra\cdots$$
such that $\Hom_{R}(\mathbb{P}, L)$ is exact for any $L\in \mathcal{L}$ and that $X\cong \ker d^{0}.$ As a consequence of Theorem \ref{thm1'}, we have the following theorem which provides a solution to the above problem.

\begin{theorem}\label{Gorenstein projective module'}
Suppose that $\Delta =\left(\begin{smallmatrix}  A & {_{A}}N_{B}\\  {_{B}}M_{A} & B \\\end{smallmatrix}\right)$ is a Morita ring with $M\oo_{A}N=0=N\oo_{B}M.$
  Let $\mathcal{C}_{1}$ (resp., $\mathcal{C}_{2}$) be a class of left (resp., right) $A$-modules and $\mathcal{D}_{1}$ (resp., $\mathcal{D}_{2}$) be a class of left (resp., right) $B$-modules such that $(\mathcal{C}_{1},\mathcal{C}_{2})$ and $(\mathcal{D}_{1},\mathcal{D}_{2})$ are complete duality pairs over $A$ and $B$ respectively, and let $N_{B}$ and $M_{A}$ be finitely generated projective.

\begin{enumerate}
 \item  Assume that $N\oo_{B}D\in \mathcal{C}_{1}$ for any $D\in \mathcal{D}_{1}.$
If $X$ is a Gorenstein $(\mathcal{C}_{1},\mathcal{C}_{2})$-projective left $A$-module, then $\mathrm{T}_{A}(X)$ is a Gorenstein $(\mathfrak{B}^{\mathcal{C}_{1}}_{\mathcal{D}_{1}},
\mathfrak{J}_{{\mathcal{C}_{2}},{\mathcal{D}_{2}}})$-projective left $\Delta$-module.

 \item  Assume that $M\oo_{A}C\in \mathcal{D}_{1}$ for any $C\in \mathcal{C}_{1}.$
If $Y$ is a Gorenstein $(\mathcal{C}_{1},\mathcal{C}_{2})$-projective left $B$-module, then $\mathrm{T}_{B}(Y)$ is a Gorenstein $(\mathfrak{B}^{\mathcal{C}_{1}}_{\mathcal{D}_{1}},
\mathfrak{J}_{{\mathcal{C}_{2}},{\mathcal{D}_{2}}})$-projective left $\Delta$-module.
\end{enumerate}

\end{theorem}

As a consequence of Theorem \ref{Gorenstein projective module'}, we provide a method to construct Ding projecive modules over the Morita ring $\Delta$ under some conditions (see Corollary \ref{ding-projective}).

The structure of this paper is organized as follows. In Section 2, we fix notations, recall some
definitions and properties of Morita rings  for later
proofs. In Section 3, we first provide a method to construct
duality pairs of $\Delta$-modules using  duality pairs of $A$-modules and $B$-modules (see Theorem \ref{thm1}), and then demonstrate that the above constructions still stand whenever the duality pairs of $A$-modules and $B$-modules are perfect or complete (see Theorem \ref{perfect} and Corollary \ref{complete}.
In Section 4, we construct Gorenstein projective modules relative to the complete duality pair over the Morita ring $\Delta$ (see Theorem \ref{Gorenstein projective module1}). This is based on the above construcuton of duality pairs over the Morita ring $\Delta$ established in Section 3. Moreover, we give an application to the construction of Ding projective modules.

\section { \bf Preliminaries}
Throughout this paper, all rings are nonzero associative ring with identity and all modules are unitary. All classes of modules are assumed to be closed under isomorphisms and contain 0.
For a ring $R$, we write $\RMod$ (resp., $\ModR$) for the category of left (resp., right) $R$-modules, and we use $_RX$ (resp., $X_R$) to denote a left (resp., right) $R$-module $X$, and the injective dimension of $_RX$ is denoted by $\mathrm{id}_RX.$
The character module $\Hom_{\mathbb{Z}}(M, \mathbb{Q}/\mathbb{Z})$ of a $R$-module $M$ is denoted by $M^{+}.$

 Let $A$ and $B$ be two rings, $_{B}M_{A}$ and $_{A}N_{B}$ bimodules with $M\oo_{A}N=0=N\oo_{B}M.$ Then $\Delta =\left(\begin{smallmatrix}  A & {_{A}}N_{B}\\  {_{B}}M_{A} & B \\\end{smallmatrix}\right)$ has a natural ring structure, where the addition is obvious and the mutiplication is given by $$\left(\begin{matrix}  a & n \\  m & b \\\end{matrix}\right)\left(\begin{matrix}  a' & n' \\  m' & b' \\\end{matrix}\right)=\left(\begin{matrix}  aa' & an'+nb' \\  ma'+bm' & bb' \\\end{matrix}\right).$$
 This is the special case of the general version of Morita rings in the sense of H. Bass \cite{B}.

In the rest of this paper, we always assume that $\Delta =\left(\begin{smallmatrix}  A & {_{A}}N_{B}\\  {_{B}}M_{A} & B \\\end{smallmatrix}\right)$ is a Morita ring with $M\oo_{A}N=0=N\oo_{B}M$.

Now we recall from \cite{eg82} the module structures over $\Delta$. Firstly, we consider the category $\mathcal{M}(\Delta)$ whose objects are tuples $(X,Y,f,g)$, where $X\in A\text{-Mod}$, $Y\in B\text{-Mod}$, $f\in \text{Hom}_{B}(M\oo_{A}X,Y)$ and $g\in \text{Hom}_{A}(N\oo_{B}Y,X)$.
 Given two objects $(X,Y,f,g), (X',Y',f',g')\in\mathcal{M}(\Delta)$, a morphism $(X,Y,f,g)\ra(X',Y',f',g')$ is a pair $(a,b)$, where $a\in \text{Hom}_{A}(X,X')$
  and $b\in \text{Hom}_{B}(Y,Y')$ such that the following diagrams commute:
\begin{align*}
 \xymatrixcolsep{0.5pc}\xymatrix{
   M\oo_{A} X \ar[d]_{1\oo a}  \ar[rrr]^{f} & & &   Y \ar[d]^{b} & \\
  M\oo_{A}X'  \ar[rrr]^{f'} & & &  Y'}
  &&
 \xymatrixcolsep{0.5pc}\xymatrix{
   N\oo_{B} Y \ar[d]_{1\oo b}  \ar[rrr]^{g} & & &   X \ar[d]^{a} & \\
  N\oo_{B}Y'  \ar[rrr]^{g'} & & &  X'.}
\end{align*}

It follows from \cite[Theorem 1.5]{eg82} that there exists an equivalence of categories $F:\mathcal{M}(\Delta)\ra \DMod$, where $F$ is defined on objects $(X,Y,f,g)$ of $\mathcal{M}(\Delta)$ as follows: $F(X,Y,f,g)=X\oplus Y$ as an abelian group, and the $\Delta$-module structure is given by
$$\left(\begin{matrix}  a & n \\  m & b \\\end{matrix}\right)\left(
                                                                    \begin{array}{c}
                                                                      x \\
                                                                      y \\
                                                                    \end{array}
                                                                  \right)
=(ax+g(n\oo y),by+f(m\oo x))$$
for any $a\in A, b\in B, n\in N,m\in M, x\in X$ and $y\in Y.$
The inverse $G:\DMod\ra\mathcal{M}(\Delta)$ of $F$ is defined by $G(Z)=(eZ,(1-e)Z,f,g)$ for any $Z\in\DMod$, where
$e=\left(
     \begin{smallmatrix}
       1 & 0 \\
       0 & 0 \\
     \end{smallmatrix}
   \right)$,
 $f:M\oo_{A}eZ\longrightarrow (1-e)Z$ and $g:N\oo_{B}(1-e)Z\longrightarrow eZ$ are the restrictions of the canonical isomorphism $\Delta\oo_{\Delta}Z\ra Z.$
More specially, $f(m\oo ez)=\left(
     \begin{smallmatrix}
       0 & 0 \\
       m & 0 \\
     \end{smallmatrix}
   \right)z$
and $g(n\oo (1-e)z)=\left(
     \begin{smallmatrix}
       0 & n \\
       0 & 0 \\
     \end{smallmatrix}
   \right)z$ for all $m\in M, n\in N$ and $z\in Z.$

 Denote by $\mathcal{N}(\Delta)$ be the category whose objects are tuples $(W,V,h,t)$, where $W\in \text{Mod-}A$, $V\in \text{Mod-}B$, $h\in \text{Hom}_{B}(W\oo_{A}N,V)$ and $t\in \text{Hom}_{A}(V\oo_{B}M,W)$. The morphisms in $\mathcal{N}(\Delta)$ are similar to those in $\mathcal{M}(\Delta)$. Moreover, there exists an equivalence of categories $G':\ModD\longrightarrow \mathcal{N}(\Delta)$ via $G'(K)=(Ke,K(1-e),\varphi,\phi)$ for any $K\in \ModD$, where
 $\varphi:Ke\oo_{A}N\longrightarrow K(1-e)$ and $\phi:K(1-e)\oo_{B}M\longrightarrow Ke$ are given by $\varphi(ke\oo n)=k\left(
     \begin{smallmatrix}
       0 & n \\
       0 & 0 \\
     \end{smallmatrix}
   \right)$
and $\phi(k(1-e)\oo m)=k\left(
     \begin{smallmatrix}
       0 & 0 \\
       m & 0 \\
     \end{smallmatrix}
   \right)$ for all $m\in M, n\in N$ and $k\in K.$

From now on, we identify left $\Delta$-modules with the objects of $\mathcal{M}(\Delta),$ and identify right $\Delta$-modules with the objects of $\mathcal{N}(\Delta).$
Given a left $\Delta$-module $(X,Y,f,g)$, we shall denote by $\widetilde{f}$ the $A$-morphism from $X$ to $\Hom_{B}(M,Y)$ given by $\widetilde{f}(x)(m)=f(m\oo x)$ for each $x\in X$ and $m\in M$, $\widetilde{g}$ the $B$-morphism from $Y$ to $\Hom_{A}(N,X)$ given by $\widetilde{g}(y)(n)=g(n\oo y)$ for each $y\in Y$ and $n\in N$.
In particular, the regular module $_{\Delta}\Delta$ corresponds to $(A,M\oo_{A}A,1,0)\ds(N\oo_{B}B,B,0,1)$.

A sequence $0\ra(X_{1},Y_{1},f_1,g_1)\xrightarrow{(a_1,b_1)} (X_{2},Y_{2},f_2,g_2)\xrightarrow{(a_2,b_2)} (X_{3},Y_{3},f_3,g_3)\ra 0$  of left $\Delta$-modules is exact if and only if $0\ra X_1\xrightarrow{a_1} X_2\xrightarrow{a_2} X_3\ra 0$ and $0\ra Y_1\xrightarrow{b_1} Y_2\xrightarrow{b_2} Y_3\ra 0$ are exact in $\AMod$ and $\BMod$ respectively.

For the Morita ring $\Delta =\left(\begin{smallmatrix}  A & {_{A}}N_{B}\\  {_{B}}M_{A} & B \\\end{smallmatrix}\right)$, we collect the following notations and facts that are used in Sections 3 and 4. Firstly, we introduce the following functors which are defined in \cite{Gao}.

(1) The functor $\mathrm{T}_{A}:\AMod\ra \DMod$ is defined by $\mathrm{T}_{A}(X)=(X,M\oo_{A}X,1,0)$ on object $X\in \AMod$, and for an $A$-morphism $a:X\ra X'$, $\mathrm{T}_{A}(a)=(a,1\oo a).$ Similarly, $\mathrm{T}_{B}:\BMod\ra \DMod$ is defined by $\mathrm{T}_{B}(Y)=(N\oo_{B}Y,Y,0,1)$ on object $Y\in \BMod$ and for a $B$-morphism $b:Y\ra Y'$, $\mathrm{T}_{B}(b)=(1\oo b,b).$

(2) The functor $\mathrm{U}_{A}:\DMod\ra \AMod$ is defined by $\mathrm{U}_{A}(X,Y,f,g)=X$ on object $(X,Y,f,g)\in \DMod$, and for a $\Delta$-morphism $(a,b):(X,Y,f,g)\ra (X',Y',f',g')$, $\mathrm{U}_{A}(a,b)=a.$ Similarly, $\mathrm{U}_{B}:\DMod\ra \BMod$ is defined by $\mathrm{U}_{B}(X,Y,f,g)=Y$ on object $(X,Y,f,g)\in \DMod$, and for a $\Delta$-morphism $(a,b):(X,Y,f,g)\ra (X',Y',f',g')$, $\mathrm{U}_{A}(a,b)=b.$

(3) The functor $\mathrm{H}_{A}:\AMod\ra \DMod$ is defined by $\mathrm{H}_{A}(X)=(X,\Hom_{A}(N,X),0,e_{X})$ on object $X\in \AMod$ and for an $A$-morphism $a:X\ra X'$, $\mathrm{H}_{A}(a)=(a,\Hom_{A}(N,a)),$ where $e_{X}:N\otimes_{B}\Hom_{A}(N,X)\ra X$ is the evaluation map.
Similarly, $\mathrm{H}_{B}:\BMod\ra \DMod$ is defined by $\mathrm{H}_{B}(Y)=(\Hom_{B}(M,Y),Y,e_{Y},0)$ on object $Y\in \BMod$, and for a $B$-morphism $b:Y\ra Y'$, $\mathrm{H}_{B}(b)=(\Hom_{B}(M,b),b),$ where $e_{X}:N\otimes_{B}\Hom_{A}(N,X)\ra X$ is the evaluation map.

\begin{lemma}\label{lem:adjoint pairs}\cite[Proposition 2.4]{Gao}
The two pairs $(\mathrm{T}_{A},\mathrm{U}_{A})$ and $(\mathrm{U}_{A},\mathrm{H}_{A})$ are adjoint pairs.
\end{lemma}

  \begin{lemma}\label{lem1}Let $\Delta =\left(\begin{smallmatrix}  A & {_{A}}N_{B}\\  {_{B}}M_{A} & B \\\end{smallmatrix}\right)$ be a Morita ring with $M\oo_{A}N=0=N\oo_{B}M$.
\begin{enumerate}
 \item \cite[Theorem 3.7.3]{KT} $(P,Q,f,g)$ is a projective left $\Delta$-module if and only if $(P,Q,f,g)=\mathrm{T}_{A}(X)\oplus\mathrm{T}_{B}(Y)=(M\oo_{A}X,1,0)\ds(N\oo_{B}Y,Y,0,1)$ for some projective left $A$-module $X$ and projective left $B$-module $Y$.

\item \cite[Corollary 3.4.7]{KT} $(P,Q,f,g)$ is an injective left $\Delta$-module if and only if $(P,Q,f,g)=\mathrm{H}_{A}(X)\oplus\mathrm{H}_{B}(Y)=(X,\Hom_{A}(N,X),0,e_{X})\ds(\Hom_{B}(M,Y),Y,e_{Y},0)$ for some injective left $A$-module $X$ and injective left $B$-module $Y$.
\end{enumerate}
\end{lemma}

\begin{lemma}\label{lem:character modules}
  Let $(X,Y,f,g)$ be a left $\Delta$-module. Then the character module of $(X,Y,f,g)$ is  $(X^{+},Y^{+},g_{+},f_{+})$, where $g_{+}:X^{+}\oo_{A}N\ra Y^{+}$ is defined by $g_{+}(s\oo n)(y)=s(g(n\oo y))$ for any $s\in X^{+}, n\in N$ and $y\in Y$, $f_{+}:Y^{+}\oo_{B}M\ra X^{+}$ is defined by $f_{+}(t\oo m)(x)=t(f(m\oo x))$ for any $t\in Y^{+}, m\in M$ and $x\in X$.
\end{lemma}
\begin{proof}

Let $(X,Y,f,g)$ be a left $\Delta$-module.
Then $F(X,Y,f,g)=X\oplus Y,$ where $F:\mathcal{M}(\Delta)\ra \DMod$ is the equivalence of categories.
Then
\begin{align*}
\Hom_{\mathbb{Z}}(F(X,Y,f,g),\mathbb{Q}/\mathbb{Z})&=\Hom_{\mathbb{Z}}(X\oplus Y,\mathbb{Q}/\mathbb{Z})\\
&=(X\oplus Y)^{+}.
\end{align*}
Since $F(X,Y,f,g)$ is a left $\Delta$-module, $(X\oplus Y)^{+}$ is a right $\Delta$-module.
Let $G':\ModD\longrightarrow\mathcal{N}(\Delta)$ be the equivalence of categories.
Then $G'((X\oplus Y)^{+})=((X\oplus Y)^{+}e,(X\oplus Y)^{+}(1-e),g',f'),$ where
 $e=\left(
     \begin{smallmatrix}
       1 & 0 \\
       0 & 0 \\
     \end{smallmatrix}
   \right), g':(X\oplus Y)^{+}e\oo_{A}N\ra (X\oplus Y)^{+}(1-e)$ and  $f':(X\oplus Y)^{+}(1-e)\oo_{B}M\ra (X\oplus Y)^{+}e$ are restrictions of the canonical isomorphism $(X\oplus Y)^{+}\oo_{\Delta}\Delta\ra X\oplus Y)^{+}.$

 We claim that $((X\oplus Y)^{+}e,(X\oplus Y)^{+}(1-e),g',f')\cong (X^{+},Y^{+},g_{+},f_{+})$, where $g_{+}:X^{+}\oo_{A}N\ra Y^{+}$ is defined by $g_{+}(s\oo n)(y)=s(g(n\oo y))$ for any $s\in X^{+}, n\in N$ and $y\in Y$, $f_{+}:Y^{+}\oo_{B}M\ra X^{+}$ is defined by $f_{+}(t\oo m)(x)=t(f(m\oo x))$ for any $t\in Y^{+}, m\in M$ and $x\in X$.

 Indeed, it is easy to check that $\theta:(X\oplus Y)^{+}e\ra X^{+}$ given by $\theta(ke)=ki_{X}$ and $\rho:(X\oplus Y)^{+}(1-e)\ra Y^{+}$ given by $\rho(k(1-e))=ki_{Y}$ are isomorphisms of right $A$-modules and right $B$-modules respectively, where $i_{X}:X\ra X\oplus Y$ and $i_{Y}:Y\ra X\oplus Y$ are canonical injections.
 We need to check that the following diagrams are commutative:
 \begin{align*}
 \xymatrixcolsep{0.5pc}\xymatrix{
   (X\oplus Y)^{+}e\oo_{A}N \ar[d]_{\theta\oo 1}  \ar[rrr]^{g'} & & &   (X\oplus Y)^{+}(1-e) \ar[d]^{\rho} & \\
  X^{+}\oo_{A}N  \ar[rrr]^{g_{+}} & & &  Y^{+}}
  \xymatrixcolsep{0.5pc}\xymatrix{
   (X\oplus Y)^{+}(1-e)\oo_{B}M \ar[d]_{\rho\oo 1}  \ar[rrr]^(.6){f'} & & &   (X\oplus Y)^{+}e \ar[d]_{\theta}& \\
  Y^{+}\oo_{B}M  \ar[rrr]^{f_{+}} & & &  X^{+},}
\end{align*}
i.e., $g_{+}(\theta\oo 1)=\rho g'$ and $f_{+}(\rho\oo 1)=\theta f'.$
For any $k\in (X\oplus Y)^{+}, y\in Y$ and $n\in N$, we have
\begin{align*}
g_{+}(\theta\oo 1)(ke\oo n)(y)&=g_{+}(ki_{X}\oo n)(y)\\
&=ki_{X}(g(n\oo y))\\
&=k\left(
                                                                    \begin{smallmatrix}
                                                                      g(n\oo y) \\
                                                                      0 \\
                                                                    \end{smallmatrix}
                                                                  \right)
\end{align*}
and
\begin{align*}
\rho g'(ke\oo n)(y)&=\rho(k\left(\begin{smallmatrix}  0 & n \\  0 & 0 \\\end{smallmatrix}\right))(y)\\
&=(k\left(\begin{smallmatrix}  0 & n \\  0 & 0 \\\end{smallmatrix}\right)i_{Y})(y)\\
&=(k\left(\begin{smallmatrix}  0 & n \\  0 & 0 \\\end{smallmatrix}\right))\left(
                                                                            \begin{smallmatrix}
                                                                              0 \\
                                                                              y \\
                                                                            \end{smallmatrix}
                                                                          \right)\\
                                                                          &=k(\left(\begin{smallmatrix}  0 & n \\  0 & 0 \\\end{smallmatrix}\right)\left(
                                                                            \begin{smallmatrix}
                                                                              0 \\
                                                                              y \\
                                                                            \end{smallmatrix}
                                                                          \right))\\
                                                                          &=k\left(
                                                                    \begin{smallmatrix}
                                                                      g(n\oo y) \\
                                                                      0 \\
                                                                    \end{smallmatrix}
                                                                  \right).
\end{align*}
Thus $g_{+}(\theta\oo 1)=\rho g'.$
Similarly, we can show $f_{+}(\rho\oo 1)=\theta f'.$

 Hence $(X,Y,f,g)^{+}\cong(X^{+},Y^{+},g_{+},f_{+}).$
\end{proof}

\bigskip
\section { \bf Duality pairs over Morita rings}

In this section, we study duality pairs over the Morita ring $\Delta$. First, we recall some definitions and facts.

\begin{definition}\cite[Definition 2.1]{HP}
A duality pair of $R$-modules is a pair $(\mathcal{X}, \mathcal{Y})$, where $\mathcal{X}$ is a class of left (resp., right) $R$-modules and $\mathcal{Y}$ is a class of right (resp., left) $R$-modules, subject to the following conditions:

\begin{enumerate}
\item For an $R$-module $X$, one has $X\in \mathcal{X}$ if and only if $X^{+}\in \mathcal{Y}.$

\item $\mathcal{Y}$ is closed under direct summands and finite direct sums.
\end{enumerate}
Recall from \cite{HP} that a duality pair $(\mathcal{X},\mathcal{Y})$ of $R$-modules is called perfect if $\mathcal{X}$ contains $R$ and is closed under direct sums and extensions.
\end{definition}

\begin{remark}\label{Remark1}
Let $(\mathcal{X},\mathcal{Y})$ be a duality pair of $R$-modules. Then $\mathcal{X}$ is closed under direct summands and finite direct sums.
\end{remark}

Let $\mathcal{C}$ be a class of left $A$-modules and $\mathcal{D}$ be a class of left $B$-modules.
We introduce the following classes of left $\Delta$-modules.

(1) $\mathfrak{A}^{\mathcal{C}}_{\mathcal{D}}=$  $\{(X,Y,f,g):X\in \mathcal{C},Y\in\mathcal{D}\}$.

(2) $\mathfrak{B}^{\mathcal{C}}_{\mathcal{D}}=$  $\{(X,Y,f,g):f~\text{and}~g\text{~are~monomorphisms,}~X/\im(g)\in \mathcal{C}~ \text{and}~Y/\im(f)\in\mathcal{D}\}$.

(3) $\mathfrak{J}^{\mathcal{C}}_{\mathcal{D}}=$  $\{(X,Y,f,g):\widetilde{f}~\text{and}~\widetilde{g}\text{~are~epimorphisms,}~\ke(\widetilde{f})\in \mathcal{C}~ \text{and}~\ke(\widetilde{g})\in\mathcal{D}\}.$

 If $\mathcal{C}$ is a class of right $A$-modules and $\mathcal{D}$ is a class of right $B$-modules, then
 there are similar symbols $\mathfrak{A}_{{\mathcal{C}},{\mathcal{D}}}$, $\mathfrak{B}_{{\mathcal{C}},{\mathcal{D}}}$,
$\mathfrak{J}_{{\mathcal{C}},{\mathcal{D}}}$ for the cases of right $\Delta$-modules.

\begin{lemma}\label{lem2}
Let $\mathcal{C}_{1}$ (resp., $\mathcal{C}_{2}$) be a class of left (resp., right) $A$-modules and $\mathcal{D}_{1}$ (resp., $\mathcal{D}_{2}$) be a class of left (resp., right) $B$-modules.
Then
\begin{enumerate}

\item[(1)] $X\in \mathcal{C}_{1}$ if and only if $\mathrm{T}_{A}(X)\in \mathfrak{B}^{\mathcal{C}_{1}}_{\mathcal{D}_{1}}$.

\item[(1')]  $Y\in \mathcal{D}_{1}$ if and only if $\mathrm{T}_{B}(Y)\in \mathfrak{B}^{\mathcal{C}_{1}}_{\mathcal{D}_{1}}$.

\item[(2)] Let $X$ be a left $A$-module. Then $X^{+}\in \mathcal{C}_{2}$ if and only if $\mathrm{T}_{A}(X)^{+}\in \mathfrak{J}_{{\mathcal{C}_{2}},{\mathcal{D}_{2}}}.$

\item[(2')] Let $Y$ be a left $B$-module. Then $Y^{+}\in \mathcal{D}_{2}$ if and only if $\mathrm{T}_{B}(Y)^{+}\in \mathfrak{J}_{{\mathcal{C}_{2}},{\mathcal{D}_{2}}}.$

\item[(3)] $X\in \mathcal{C}_{1}$ if and only if $\mathrm{H}_{A}(X)\in \mathfrak{J}^{\mathcal{C}_{1}}_{\mathcal{D}_{1}}.$

\item[(3')] $Y\in \mathcal{D}_{1}$ if and only if $\mathrm{H}_{B}(Y)\in \mathfrak{J}^{\mathcal{C}_{1}}_{\mathcal{D}_{1}}.$

\item[(4)] Let $X$ be a left $A$-module. If $_{A}N$ is finite presented, then $X^{+}\in \mathcal{C}_{2}$ if and only if $\mathrm{H}_{A}(X)^{+}\in \mathfrak{B}_{{\mathcal{C}_{2}},{\mathcal{D}_{2}}}.$

\item[(4')] Let $Y$ be a left $B$-module. If $_{B}M$ is finitely generated, then $Y^{+}\in \mathcal{D}_{2}$ if and only if $\mathrm{H}_{B}(Y)^{+}\in \mathfrak{B}_{{\mathcal{C}_{2}},{\mathcal{D}_{2}}}.$
\end{enumerate}
\end{lemma}

\begin{proof} We only prove (1)-(4), and the proofs of (1')-(4') are similar.

(1) Since $N\oo_{B}M=0$ by hypothesis, it is easy to check that $X\in \mathcal{C}_{1}$ if and only if $\mathrm{T}_{A}(X)\in \mathfrak{B}^{\mathcal{C}_{1}}_{\mathcal{D}_{1}}$.

(2) Since $N\oo_{B}M=0$ by hypothesis, $\Hom_{B}(N,(M\oo_{A}X)^{+})\cong (N\oo_{B}M\oo_{A}X)^{+}=0.$
Then it follows that $X^{+}\in \mathcal{C}_{2}$ if and only if $\mathrm{T}_{A}(X)^{+}\in \mathfrak{J}_{{\mathcal{C}_{2}},{\mathcal{D}_{2}}}.$

(3) Since $N\oo_{B}M=0$ by hypothesis, $\Hom_{B}(M,\Hom_{A}(N,X))\cong \Hom_{A}(N\oo_{B}M,X)=0.$
Then it follows that $X\in \mathcal{C}_{1}$ if and only if $\mathrm{H}_{A}(X)\in \mathfrak{J}^{\mathcal{C}_{1}}_{\mathcal{D}_{1}}.$

(4) Since $_{A}N$ is finitely presented, $X^{+}\oo_{A}N\cong \Hom_{A}(N,X)^{+}$ by \cite[Theorem 3.2.11]{EJ}.
Then $(e_{X})_{+}:X^{+}\oo_{A}N\ra \Hom_{A}(N,X)^{+}$ is an isomorphism,
where $e_{X}:N\otimes_{B}\Hom_{A}(N,X)\ra X$ is the evaluation map.
By \cite[Theorem 3.2.11]{EJ} again, $\Hom_{A}(N,X)^{+}\oo_{A}M\cong X^{+}\oo_{A}N\oo_{B}M=0$ since $N\oo_{B}M=0$ by hypothesis.
Then it is easy to see that $X^{+}\in \mathcal{C}_{2}$ if and only if $\mathrm{H}_{A}(X)^{+}\in \mathfrak{B}_{{\mathcal{C}_{2}},{\mathcal{D}_{2}}}.$
\end{proof}

The following result  gives a method to construct duality pairs over the Morita ring $\Delta$ which contains Theorem \ref{thm1'} in the introduction.
\begin{theorem}\label{thm1} Let $\mathcal{C}_{1}$ (resp., $\mathcal{C}_{2}$) be a class of left (resp., right) $A$-modules and $\mathcal{D}_{1}$ (resp., $\mathcal{D}_{2}$) be a class of left (resp., right) $B$-modules. The following statements are equivalent:

\begin{enumerate}

\item $(\mathcal{C}_{1},\mathcal{C}_{2})$ and $(\mathcal{D}_{1},\mathcal{D}_{2})$ are duality pairs.

\item $(\mathfrak{B}^{\mathcal{C}_{1}}_{\mathcal{D}_{1}},
\mathfrak{J}_{{\mathcal{C}_{2}}, {\mathcal{D}_{2}}})$ is a duality pair.

\item $(\mathfrak{A}^{\mathcal{C}_{1}}_{\mathcal{D}_{1}},
\mathfrak{A}_{{\mathcal{C}_{2}},{\mathcal{D}_{2}}})$ is a duality pair.
\end{enumerate}

Moreover, if $_{B}M$ and $_{A}N$ are finitely presented, then the above statements are equivalent to
\begin{enumerate}

\item[(4)] $(\mathfrak{J}^{\mathcal{C}_{1}}_{\mathcal{D}_{1}},
\mathfrak{B}_{{\mathcal{C}_{2}},{\mathcal{D}_{2}}})$ is a duality pair.
\end{enumerate}
\end{theorem}

\begin{proof}
$(1)\Rightarrow (2).$ Let $(X,Y,f,g)$ be a left $\Delta$-module.
Then $(X,Y,f,g)^{+}\cong(X^{+},Y^{+},g_{+},f_{+})$ by Lemma \ref{lem:character modules}.

 If $(X,Y,f,g)\in \mathfrak{B}^{\mathcal{C}_{1}}_{\mathcal{D}_{1}}$,
 then we have the following exact sequences

$$0\ra M\oo_{A}X\xrightarrow{f}Y\ra Y/\im(f)\ra 0$$ with $Y/\im(f)\in \mathcal{D}_{1}$ and
$$0\ra N\oo_{B}Y\xrightarrow{g}X\ra X/\im(g)\ra 0$$ with $X/\im(g)\in \mathcal{C}_{1}$, which induce the exact sequences
$$0\ra {(Y/\im(f))}^{+}\ra Y^{+}\xrightarrow{f^{+}}(M\oo_{A}X)^{+}\ra 0$$ and
$$0\ra {(X/\im(g))}^{+}\ra X^{+}\xrightarrow{g^{+}}(N\oo_{B}Y)^{+}\ra 0.$$
Hence we have the following commutative diagrams
$$\xymatrix{0\ar[r]&{(Y/\im(f))}^{+}\ar[r]^{}\ar[d]_{}&Y^{+}\ar[r]^(.4){f^{+}}\ar@{=}[d]&(M\oo_{A}X)^{+}\ar[r]\ar[d]_{\cong}&0\\
 0\ar[r]&\ke(\widetilde{f_{+}})\ar[r]^{}&Y^{+}\ar[r]^(.3){\widetilde{f_{+}}}&\Hom_{A}(M,X^{+})\ar[r]&0}$$
and
$$\xymatrix{0\ar[r]&{(X/\im(g))}^{+}\ar[r]^{}\ar[d]_{}&X^{+}\ar[r]^(.4){g^{+}}\ar@{=}[d]&(N\oo_{B}Y)^{+}\ar[r]\ar[d]_{\cong}&0\\
 0\ar[r]&\ke(\widetilde{g_{+}})\ar[r]^{}&X^{+}\ar[r]^(.3){\widetilde{g_{+}}}&\Hom_{B}(N,Y^{+})\ar[r]&0.}$$

 Thus $\ke(\widetilde{f_{+}})\cong{(Y/\im(f))}^{+}\in \mathcal{D}_{2}$ and $\ke(\widetilde{g_{+}})\cong{(X/\im(g))}^{+}\in \mathcal{C}_{2}$. Hence $(X^{+},Y^{+},g_{+},f_{+})\in \mathfrak{J}_{{\mathcal{C}_{2}},{\mathcal{D}_{2}}}.$

Conversely, if $(X,Y,f,g)^{+}=(X^{+},Y^{+},g_{+},f_{+})\in \mathfrak{J}_{{\mathcal{C}_{2}},{\mathcal{D}_{2}}}$, then we get the exact sequences
 $$0\ra \ke(\widetilde{f_{+}})\ra Y^{+}\xrightarrow{\widetilde{f_{+}}}\Hom_{A}(M,X^{+})\ra 0$$
and
$$0\ra \ke(\widetilde{g_{+}})\ra X^{+}\xrightarrow{\widetilde{g_{+}}}\Hom_{B}(N,Y^{+})\ra 0$$
with $\ke(\widetilde{f_{+}})\in \mathcal{D}_{2}$ and $\ke(\widetilde{g_{+}})\in \mathcal{C}_{2}$.
Hence by the above the commutative diagrams, we have $\ke(\widetilde{f_{+}})\cong{(Y/\im(f))}^{+}\in \mathcal{D}_{2}$ and $\ke(\widetilde{g_{+}})\cong{(X/\im(g))}^{+}\in \mathcal{C}_{2}$.
 Hence $X/\im(g)\in \mathcal{C}_{1}$ and $Y/\im(f)\in \mathcal{D}_{1}$. Thus
 $(X,Y,f,g)\in \mathfrak{B}^{\mathcal{C}_{1}}_{\mathcal{D}_{1}}.$

 It is easy to check that $\mathfrak{J}_{{\mathcal{C}_{2}}, {\mathcal{D}_{2}}}$
 is closed under direct summands and finite direct sums.
 So $(\mathfrak{B}^{\mathcal{C}_{1}}_{\mathcal{D}_{1}},
\mathfrak{J}_{{\mathcal{C}_{2}}, {\mathcal{D}_{2}}})$ is a duality pair.

 $(2)\Rightarrow (1).$
Note that
 $X\in \mathcal{C}_{1}$ if and only if $\mathrm{T}_{A}(X)\in \mathfrak{B}^{\mathcal{C}_{1}}_{\mathcal{D}_{1}}$ by Lemma \ref{lem2}(1).
Since $(\mathfrak{B}^{\mathcal{C}_{1}}_{\mathcal{D}_{1}},
\mathfrak{J}_{{\mathcal{C}_{2}}, {\mathcal{D}_{2}}})$ is a duality pair,
$\mathrm{T}_{A}(X)\in \mathfrak{B}^{\mathcal{C}_{1}}_{\mathcal{D}_{1}}$
  if and only if $\mathrm{T}_{A}(X)^{+}\in \mathfrak{J}_{{\mathcal{C}_{2}},{\mathcal{D}_{2}}}$.
 It follows from Lemma \ref{lem2}(2) that $X\in \mathcal{C}_{1}$
   if and only if $X^{+}\in \mathcal{C}_{2}$.

 We now show that $\mathcal{C}_{2}$ is closed under direct summands and finite direct sums.
 Note that $X\in \mathcal{C}_{2}$ if and only if $(X,\Hom_{A}(M,X),0,e_{X})\in \mathfrak{J}_{{\mathcal{C}_{2}},{\mathcal{D}_{2}}}$ by the similar proof of Lemma \ref{lem2}(3), where $e_{X}:\Hom_{A}(M,X)\oo_{B}M\ra X$ is the evaluation map.
 Let $X_{1}\oplus X_{2}\in \mathcal{C}_{2}.$
  By the above argument, we have $(X_{1}\oplus X_{2},\Hom_{A}(M,X_{1}\oplus X_{2}),0,e_{X_{1}\oplus X_{2}})\in \mathfrak{J}_{{\mathcal{C}_{2}},{\mathcal{D}_{2}}}.$
   Note that $(X_{1},\Hom_{A}(M,X_{1}),0,e_{X_{1}})\oplus(X_{2},\Hom_{A}(M,X_{2}),0,e_{X_{2}})\cong(X_{1}\oplus X_{2},\Hom_{A}(M,X_{1}\oplus X_{2}),0,e_{X_{1}\oplus X_{2}}).$
Since $\mathfrak{J}_{{\mathcal{C}_{2}},{\mathcal{D}_{2}}}$ is closed under direct summands, $(X_{1},\Hom_{A}(M,X_{1}),0,e_{X_{1}})$ and $(X_{2},\Hom_{A}(M,X_{2}),0,e_{X_{2}})$ are in $\mathfrak{J}_{{\mathcal{C}_{2}},{\mathcal{D}_{2}}}.$
So $X_{1}$ and $X_{2}$ are in $\mathcal{C}_{2}.$
 It follows that $\mathcal{C}_{2}$ is closed under direct summands.
  Similarly, we can show that $\mathcal{C}_{2}$ is closed under finite direct sums. So $(\mathcal{C}_{1},\mathcal{C}_{2})$
 is a duality pair.

  Note that $Y\in \mathcal{D}_{1}$ if and only if $\mathrm{T}_{B}(Y)\in \mathfrak{B}^{\mathcal{C}_{1}}_{\mathcal{D}_{1}}$ by Lemma \ref{lem2}(1').
 Since $(\mathfrak{B}^{\mathcal{C}_{1}}_{\mathcal{D}_{1}},
\mathfrak{J}_{{\mathcal{C}_{2}}, {\mathcal{D}_{2}}})$ is a duality pair,
 $\mathrm{T}_{B}(Y)\in \mathfrak{B}^{\mathcal{C}_{1}}_{\mathcal{D}_{1}}$
  if and only if $\mathrm{T}_{B}(Y)^{+}\in \mathfrak{J}_{\mathcal{C}_{2},\mathcal{D}_{2}}$.
It follows from Lemma \ref{lem2}(2') that $Y\in \mathcal{D}_{1}$
   if and only if $Y^{+}\in \mathcal{D}_{2}$. We can prove that $\mathcal{D}_{2}$ is closed under direct summands and finite direct sums by the similar proof of that $\mathcal{C}_{2}$ is closed under direct summands and finite direct sums. So $(\mathcal{D}_{1},\mathcal{D}_{2})$
 is a duality pair.

 $(1)\Rightarrow (3).$ Let $(X,Y,f,g)$ be a left $\Delta$-module.
  Then $(X,Y,f,g)\in \mathfrak{A}^{\mathcal{C}_{1}}_{\mathcal{D}_{1}}$ if and only if $X\in \mathcal{C}_{1}$ and $Y\in \mathcal{D}_{1}$ if and only if
$(X^{+},Y^{+},g_{+},f_{+})\in \mathfrak{A}_{{\mathcal{C}_{2}},{\mathcal{D}_{2}}}$.
Note that $(X,Y,f,g)^{+}\cong(X^{+},Y^{+},g_{+},f_{+})$ by Lemma \ref{lem:character modules}.
It is easy to check that $\mathfrak{A}_{{\mathcal{C}_{2}},{\mathcal{D}_{2}}}$ is closed under direct summands and finite direct sums.
Hence $(\mathfrak{A}^{\mathcal{C}_{1}}_{\mathcal{D}_{1}},
\mathfrak{A}_{{\mathcal{C}_{2}},{\mathcal{D}_{2}}})$ is a duality pair.

 $(3)\Rightarrow (1).$
 It is straightforward to check that $X\in \mathcal{C}_{1}$ if and only if $(X,0,0,0)\in\mathfrak{A}^{\mathcal{C}_{1}}_{\mathcal{D}_{1}}$ if and only if $(X^{+},0,0,0)\in \mathfrak{A}_{\mathcal{C}_{2},\mathcal{D}_{2}}$ if and only if $X^{+}\in \mathcal{C}_{2}$. It is easy to see that $\mathcal{C}_{2}$ is closed under direct summands and finite direct sums.
 So $(\mathcal{C}_{1},\mathcal{C}_{2})$ is a duality pair.

It is straightforward to check that $Y\in \mathcal{D}_{1}$ if and only if $(0,Y,0,0)\in\mathfrak{A}^{\mathcal{C}_{1}}_{\mathcal{D}_{1}}$ if and only if $(0,Y^{+},0,0)\in \mathfrak{A}_{\mathcal{C}_{2},\mathcal{D}_{2}}$ if and only if $Y^{+}\in \mathcal{D}_{2}$. It is easy to see that $\mathcal{D}_{2}$ is closed under direct summands and finite direct sums. So $(\mathcal{D}_{1},\mathcal{D}_{2})$ is a duality pair.

Now, we suppose that $_{B}M$ and $_{A}N$ are finitely presented.

For a right $A$-module $C$ and right $B$-module $D$, it should be noted that $(C,C\oo_{A}N,1,0)\in \mathfrak{B}_{\mathcal{C}_{2},\mathcal{D}_{2}}$ if and only if $C\in \mathcal{C}_{2}$, and $(D\oo_{B}M,D,0,1)\in \mathfrak{B}_{\mathcal{C}_{2},\mathcal{D}_{2}}$ if and only if $D\in \mathcal{D}_{2}$
by the similar proof of Lemma \ref{lem2}(1).
Thus it is easy to check that $\mathcal{C}_{2}$ and $\mathcal{D}_{2}$ are closed under direct summands and finite direct sums if and only if  $\mathfrak{B}_{\mathcal{C}_{2},\mathcal{D}_{2}}$ is closed under direct summands and finite direct sums.

 $(1)\Rightarrow (4).$ Let $(X,Y,f,g)$ be a left $\Delta$-module.
 If $(X,Y,f,g)\in \mathfrak{J}^{\mathcal{C}_{1}}_{\mathcal{D}_{1}}$, then we have  exact sequences
 $$0\ra \ke\widetilde{f}\ra X\xrightarrow{\widetilde{f}}\Hom_{B}(M,Y)\ra0$$
 and
 $$0\ra \ke\widetilde{g}\ra Y\xrightarrow{\widetilde{g}}\Hom_{A}(N,X)\ra0$$
with $\ke\widetilde{f}\in \mathcal{C}_{1}$ and $\ke\widetilde{g}\in \mathcal{D}_{1}$, which induce the exact sequences
$$0\ra \Hom_{B}(M,Y)^{+}\ra X^{+}\ra (\ke\widetilde{f})^{+}\ra 0$$
and
$$0\ra \Hom_{A}(N,X)^{+}\ra Y^{+}\ra (\ke\widetilde{g})^{+}\ra 0.$$
So $(\ke\widetilde{f})^{+}\in\mathcal{C}_{2}$ and $(\ke\widetilde{g})^{+}\in \mathcal{D}_{2}.$
Since $_{B}M$ and $_{A}N$ are finitely presented, we have the following commutative diagrams by \cite[Theorem 3.2.11]{EJ}:
$$\xymatrix{0\ar[r]&Y^{+}\oo_{B}M\ar[r]^{f_{+}}\ar[d]_{\cong}&X^{+}\ar[r]^{}\ar@{=}[d]&X^{+}/\im(f_{+})\ar[r]\ar[d]_{}&0\\
 0\ar[r]&\Hom_{B}(M,Y)^{+}\ar[r]^{}&X^{+}\ar[r]&(\ke(\widetilde{f}))^{+}\ar[r]&0}$$
and
$$\xymatrix{0\ar[r]&X^{+}\oo_{A}N\ar[r]^{g_{+}}\ar[d]_{\cong}&Y^{+}\ar[r]^{}\ar@{=}[d]&Y^{+}/\im(g_{+})\ar[r]\ar[d]_{}&0\\
 0\ar[r]&\Hom_{A}(N,X)^{+}\ar[r]^{}&Y^{+}\ar[r]&(\ke(\widetilde{g}))^{+}\ar[r]&0.}$$
Thus $(X,Y,f,g)^{+}=(X^{+},Y^{+},g_{+},f_{+})\in \mathfrak{B}_{{\mathcal{C}_{2}},{\mathcal{D}_{2}}}.$

Conversely, if $(X,Y,f,g)^{+}=(X^{+},Y^{+},g_{+},f_{+})\in \mathfrak{B}_{{\mathcal{C}_{2}},{\mathcal{D}_{2}}}$, then we have exact sequences
$$0\ra Y^{+}\oo_{B}M\xrightarrow{f_{+}}X^{+}\ra X^{+}/\im(f_{+})\ra 0$$
and
$$0\ra X^{+}\oo_{A}N\xrightarrow{g_{+}}Y^{+}\ra Y^{+}/\im(g_{+})\ra 0$$
with $X^{+}/\im(f_{+})\in \mathcal{C}_{2}$ and $Y^{+}/\im(g_{+})\in \mathcal{D}_{2}$. By the above diagrams, we get exact sequences
$$0\ra\ke(\widetilde{f})\ra X\xrightarrow{\widetilde{f}}\Hom_{B}(M,Y)\ra 0$$
and
$$0\ra\ke(\widetilde{g})\ra Y\xrightarrow{\widetilde{g}}\Hom_{A}(N,X)\ra 0$$
with $(\ke(\widetilde{f}))^{+}\cong X^{+}/\im(f_{+})\in \mathcal{C}_{2}$ and $(\ke(\widetilde{g}))^{+}\cong Y^{+}/\im(g_{+})\in \mathcal{D}_{2}$.
Hence $(X,Y,f,g)\in \mathfrak{J}^{\mathcal{C}_{1}}_{\mathcal{D}_{1}}$.
So $(\mathfrak{J}^{\mathcal{C}_{1}}_{\mathcal{D}_{1}},
\mathfrak{B}_{{\mathcal{C}_{2}},{\mathcal{D}_{2}}})$ is a duality pair.

$(4)\Rightarrow (1).$ Note that $X\in \mathcal{C}_{1}$ if and only if $\mathrm{H}_{A}(X)\in \mathfrak{J}^{\mathcal{C}_{1}}_{\mathcal{D}_{1}}$ by Lemma \ref{lem2}(3).
Since $(\mathfrak{J}^{\mathcal{C}_{1}}_{\mathcal{D}_{1}},
\mathfrak{B}_{{\mathcal{C}_{2}},{\mathcal{D}_{2}}})$ is a duality pair,
$\mathrm{H}_{A}(X)\in \mathfrak{J}^{\mathcal{C}_{1}}_{\mathcal{D}_{1}}$
 if and only if $\mathrm{H}_{A}(X)^{+}\in \mathfrak{B}_{{\mathcal{C}_{2}},{\mathcal{D}_{2}}},$
Since $_{A}N$ is finitely presented, it follows from Lemma \ref{lem2}(4) that $\mathrm{H}_{A}(X)^{+}\in \mathfrak{B}_{{\mathcal{C}_{2}},{\mathcal{D}_{2}}}$ if and only if $X^{+}\in \mathcal{C}_{2}$.
Thus $X\in \mathcal{C}_{1}$ if and only if $X^{+}\in \mathcal{C}_{2}$.
 Hence $(\mathcal{C}_{1},\mathcal{C}_{2})$ is a duality pair.

Similarly, we can show that $(\mathcal{D}_{1},\mathcal{D}_{2})$ is a duality pair.
\end{proof}

Let $\mathcal{X}$ be a class of left $R$-modules over a ring $R$. Recall from \cite{E} that an $\mathcal{X}$-precover of a left $R$-module $N$ is an $R$-homomorphism $\varphi:X\rightarrow N$ with $X\in\mathcal{X}$ such that, for every $R$-homomorphism $\varphi':X'\rightarrow N$, where $X'\in \mathcal{X}$, there exists an $R$-morphism $\psi:X'\rightarrow X$ with $\varphi'=\varphi\psi.$
An $\mathcal{X}$-precover $\varphi:X\rightarrow N$ is an $\mathcal{X}$-cover if every $R$-morphism $\psi:X\rightarrow X$ satisfying $\varphi=\varphi\psi$ is an automorphism. The class $\mathcal{X}$ is called (pre)covering if every $R$-module has an $\mathcal{X}$-(pre)cover.
Dually we have the definitions of $\mathcal{X}$-(pre)envelope
and $\mathcal{X}$-(pre)enveloping.

\begin{corollary}
Let $\mathcal{C}_{1}$ (resp., $\mathcal{C}_{2}$) be a class of left (resp., right) $A$-modules and $\mathcal{D}_{1}$ (resp., $\mathcal{D}_{2}$) be a class of left (resp., right) $B$-modules such that
 $(\mathcal{C}_{1},\mathcal{C}_{2})$ and $(\mathcal{D}_{1},\mathcal{D}_{2})$ are duality pairs.

\begin{enumerate}
\item If $\mathcal{C}_{1}$ and $\mathcal{D}_{1}$ are closed under direct sums, then every left $\Delta$-module has a $\mathfrak{B}^{\mathcal{C}_{1}}_{\mathcal{D}_{1}}$-cover and $\mathfrak{A}^{\mathcal{C}_{1}}_{\mathcal{D}_{1}}$-cover.
\item If $\mathcal{C}_{1}$ and $\mathcal{D}_{1}$ are closed under direct products, then every left $\Delta$-module has a $\mathfrak{A}^{\mathcal{C}_{1}}_{\mathcal{D}_{1}}$-preenvelope. Moreover, if $_{B}M$ and $_{A}N$ are finitely presented, then every left $\Delta$-module has a $\mathfrak{J}^{\mathcal{C}_{1}}_{\mathcal{D}_{1}}$-preenvelope.
If $M_{A}$ and $N_{B}$ are finitely presented, then every left $\Delta$-module has a $\mathfrak{B}^{\mathcal{C}_{1}}_{\mathcal{D}_{1}}$-preenvelope.
\end{enumerate}
\end{corollary}

\begin{proof}
By Theorem \ref{thm1}, $(\mathfrak{B}^{\mathcal{C}_{1}}_{\mathcal{D}_{1}},
\mathfrak{J}_{{\mathcal{C}_{2}},{\mathcal{D}_{2}}})$ and $(\mathfrak{A}^{\mathcal{C}_{1}}_{\mathcal{D}_{1}},
\mathfrak{A}_{{\mathcal{C}_{2}},{\mathcal{D}_{2}}})$ are duality pairs.

(1). Since $\mathcal{C}_{1}$ and $\mathcal{D}_{1}$ are closed under direct sums, it is easy to check that $\mathfrak{B}^{\mathcal{C}_{1}}_{\mathcal{D}_{1}}$ and $\mathfrak{A}^{\mathcal{C}_{1}}_{\mathcal{D}_{1}}$ are closed under direct sums.
Hence every left $\Delta$-module has a $\mathfrak{B}^{\mathcal{C}_{1}}_{\mathcal{D}_{1}}$-cover and $\mathfrak{A}^{\mathcal{C}_{1}}_{\mathcal{D}_{1}}$-cover by \cite[Theorem 3.1]{HP}.

(2). Since $\mathcal{C}_{1}$ and $\mathcal{D}_{1}$ are closed under direct products,
it is easy to check that $\mathfrak{A}^{\mathcal{C}_{1}}_{\mathcal{D}_{1}}$ is closed under direct products.
So every left $\Delta$-modules has a $\mathfrak{A}^{\mathcal{C}_{1}}_{\mathcal{D}_{1}}$-preenvelope by \cite[Theorem 3.1]{HP}.

 If
$_{B}M$ and $_{A}N$ are finitely presented, then  $(\mathfrak{J}^{\mathcal{C}_{1}}_{\mathcal{D}_{1}},
\mathfrak{B}_{{\mathcal{C}_{2}},{\mathcal{D}_{2}}})$ is a duality pair.
Since $\mathcal{C}_{1}$ and $\mathcal{D}_{1}$ are closed under direct products,
it is easy to check that $\mathfrak{J}^{\mathcal{C}_{1}}_{\mathcal{D}_{1}}$ is closed under direct products.
So every left $\bigtriangleup$-module has a $\mathfrak{J}^{\mathcal{C}_{1}}_{\mathcal{D}_{1}}$-preenvelope by \cite[Theorem 3.1]{HP}.

If
$M_{A}$ and $N_{B}$ are finitely presented, then $M\oo_{A}\prod_{i\in I} X_{i}\cong \prod_{i\in I} M\oo_{A}X_{i}$ for any family of left $A$-modules $\{X_{i}\}_{i\in I}$ and $N\oo_{B}\prod_{i\in I} Y_{i}\cong \prod_{i\in I} N\oo_{B}Y_{i}$ for any family of left $B$-modules $\{Y_{i}\}_{i\in I}$ by \cite[Lemma 3.2.21]{EJ}.
Thus it is easy to check that $\mathfrak{B}^{\mathcal{C}_{1}}_{\mathcal{D}_{1}}$
is closed under direct products.
So every left $\Delta$-module has a $\mathfrak{B}^{\mathcal{C}_{1}}_{\mathcal{D}_{1}}$-preenvelope by \cite[Theorem 3.1]{HP}.
\end{proof}

Recall that a duality pair $(\mathcal{X},\mathcal{Y})$ of $R$-modules is symmetric \cite{Gillespie} if both $(\mathcal{X},\mathcal{Y})$ and $(\mathcal{Y},\mathcal{X})$
are duality pairs.

\begin{corollary}\label{symmetric} Let $\mathcal{C}_{1}$ (resp., $\mathcal{C}_{2}$) be a class of left (resp., right) $A$-modules and $\mathcal{D}_{1}$ (resp., $\mathcal{D}_{2}$) be a class of left (resp., right) $B$-modules. The following statements are equivalent:

\begin{enumerate}
\item $(\mathfrak{A}^{\mathcal{C}_{1}}_{\mathcal{D}_{1}},
\mathfrak{A}_{{\mathcal{C}_{2}},{\mathcal{D}_{2}}})$ is a symmetric duality pair if and only if $(\mathcal{C}_{1},\mathcal{C}_{2})$ and $(\mathcal{D}_{1},\mathcal{D}_{2})$ are symmetric duality pairs.

\item Suppose that $M_{A}$ and $N_{B}$ are finitely presented, then $(\mathfrak{B}^{\mathcal{C}_{1}}_{\mathcal{D}_{1}},
\mathfrak{J}_{{\mathcal{C}_{2}},{\mathcal{D}_{2}}})$ is a symmetric duality pair if and only if $(\mathcal{C}_{1},\mathcal{C}_{2})$ and $(\mathcal{D}_{1},\mathcal{D}_{2})$ are symmetric duality pairs.

\item Suppose that $_{B}M$ and $_{A}N$ are finitely presented, then $(\mathfrak{J}^{\mathcal{C}_{1}}_{\mathcal{D}_{1}},
\mathfrak{B}_{{\mathcal{C}_{2}},{\mathcal{D}_{2}}})$ is a symmetric duality pair if and only if $(\mathcal{C}_{1},\mathcal{C}_{2})$ and $(\mathcal{D}_{1},\mathcal{D}_{2})$ are symmetric duality pairs.
\end{enumerate}
\end{corollary}
\begin{proof}
It follows from Theorem \ref{thm1}.
\end{proof}

Perfect duality pairs can produce perfect cotorsion pairs by \cite{HP}. In the following, we consider when a duality pairs over $\Delta$ is perfect.

\begin{theorem}\label{perfect} Let $\mathcal{C}_{1}$ (resp., $\mathcal{C}_{2}$) be a class of left (resp., right) $A$-modules and $\mathcal{D}_{1}$ (resp., $\mathcal{D}_{2}$) be a class of left (resp., right) $B$-modules.
\begin{enumerate}
\item  Suppose that $\Tor_{1}^{B}(N,D)=0=\Tor_{1}^{A}(M,C)$ for any $D\in \mathcal{D}_{1}$ and any $C\in \mathcal{C}_{1}$.
    Then
    $(\mathfrak{B}^{\mathcal{C}_{1}}_{\mathcal{D}_{1}},
\mathfrak{J}_{{\mathcal{C}_{2}},{\mathcal{D}_{2}}})$ is a perfect duality pairs if and only if $(\mathcal{C}_{1},\mathcal{C}_{2})$ and $(\mathcal{D}_{1},\mathcal{D}_{2})$ are perfect duality pairs.

\item
$(\mathfrak{A}^{\mathcal{C}_{1}}_{\mathcal{D}_{1}},\mathfrak{A}_{{\mathcal{C}_{2}},{\mathcal{D}_{2}}})$ is a perfect duality pair
if and only if $_{B}M\in \mathcal{D}_{1}$ and $_{A}N\in \mathcal{C}_{1},$ $(\mathcal{C}_{1},\mathcal{C}_{2})$ and $(\mathcal{D}_{1},\mathcal{D}_{2})$ are perfect duality pairs.
\end{enumerate}
\end{theorem}

\begin{proof}
(1) $``\Rightarrow"$. Note that $_{\Delta}\Delta=\mathrm{T}_{A}(A)\oplus\mathrm{T}_{B}(B)$. Since $_{\Delta}\Delta\in \mathfrak{B}^{\mathcal{C}_{1}}_{\mathcal{D}_{1}}$, we have $\mathrm{T}_{A}(A)$ and $\mathrm{T}_{B}(B)$ are in $\mathfrak{B}^{\mathcal{C}_{1}}_{\mathcal{D}_{1}}$ by Remark \ref{Remark1}.
It follows from Lemma \ref{lem2} that $A\in \mathcal{C}_{1}$ and $B\in \mathcal{D}_{1}.$

Let $\{X_{i}\}_{i\in I}$ be a family of left $A$-modules with $X_{i}\in \mathcal{C}_{1}.$
Then $\mathrm{T}_{A}(X_{i})\in \mathfrak{B}^{\mathcal{C}_{1}}_{\mathcal{D}_{1}}$.
Hence $\ds_{i\in I}\mathrm{T}_{A}(X_{i})\cong\mathrm{T}_{A}(\ds_{i\in I}X_{i})\in\mathfrak{B}^{\mathcal{C}_{1}}_{\mathcal{D}_{1}}.$
Thus $\ds_{i\in I}X_{i}\in \mathcal{C}_{1}$ by Lemma \ref{lem2}(1).

Let $$0\ra X_{1}\ra X_{2}\ra X_{3}\ra 0$$ be an exact sequence of left $A$-modules with $X_{1}\in \mathcal{C}_{1}$ and $X_{3}\in \mathcal{C}_{1}.$
Since $\Tor_{1}^{A}(M,C)=0$ for any $C\in \mathcal{C}_{1}$, we have the following exact sequence of left $\Delta$-modules
$$0\ra \mathrm{T}_{A}(X_{1})\ra \mathrm{T}_{A}(X_{2})\ra \mathrm{T}_{A}(X_{3})\ra 0.$$
Since $\mathrm{T}_{A}(X_{1})$ and $\mathrm{T}_{A}(X_{3})$ are in $\mathfrak{B}^{\mathcal{C}_{1}}_{\mathcal{D}_{1}},$ $\mathrm{T}_{A}(X_{2})\in \mathfrak{B}^{\mathcal{C}_{1}}_{\mathcal{D}_{1}}$ as $\mathfrak{B}^{\mathcal{C}_{1}}_{\mathcal{D}_{1}}$ is closed under extensions.
It follows from Lemma \ref{lem2}(1) that $X_{2}\in \mathcal{C}_{1}.$
Hence $(\mathcal{C}_{1},\mathcal{C}_{2})$ is a perfect duality pair.

Similarly, we can show that $(\mathcal{D}_{1},\mathcal{D}_{2})$ is a perfect duality pair.

$``\Leftarrow"$.  Since $A\in \mathcal{C}_{1}$ and $B\in \mathcal{D}_{1}$, we have $$_{\Delta}\Delta=\mathrm{T}_{A}(A)\oplus\mathrm{T}_{B}(B)\in \mathfrak{B}^{\mathcal{C}_{1}}_{\mathcal{D}_{1}}.$$
Let $\{(X_{i},Y_{i},f_{i},g_{i})\}_{i\in I}$ be a family of left $\Delta$-modules with $(X_{i},Y_{i},f_{i},g_{i})\in \mathfrak{B}^{\mathcal{C}_{1}}_{\mathcal{D}_{1}}.$
Then we get the following exact sequences
$$0\ra M\oo_{A}X_{i}\xrightarrow{f_{i}}Y_{i}\ra Y_{i}/\im(f_{i})\ra 0$$
and
$$0\ra N\oo_{B}Y_{i}\xrightarrow{g_{i}}X_{i}\ra X_{i}/\im(g_{i})\ra 0$$
with $Y_{i}/\im(f_{i})\in \mathcal{D}_{1}$ and $X_{i}/\im(g_{i})\in \mathcal{C}_{1}$.
Consider the following commutative diagrams

$$\xymatrix{0\ar[r]&M\oo_{A}(\ds_{i\in I}X_{i})\ar[r]^{}\ar[d]_{\cong}&\ds_{i\in I}Y_{i}\ar[r]^{}\ar@{=}[d]&\ds_{i\in I}Y_{i}/\im(\ds_{i\in I}f_{i})\ar[r]\ar[d]_{}&0\\
 0\ar[r]&\ds_{i\in I}(M\oo_{A}X_{i})\ar[r]^{}&\ds_{i\in I}Y_{i}\ar[r]&\ds_{i\in I}(Y_{i}/\im(f_{i}))\ar[r]&0}$$
and
$$\xymatrix{0\ar[r]&N\oo_{B}(\ds_{i\in I}Y_{i})\ar[r]^{}\ar[d]_{\cong}&\ds_{i\in I}X_{i}\ar[r]^{}\ar@{=}[d]&\ds_{i\in I}X_{i}/\im(\ds_{i\in I}g_{i})\ar[r]\ar[d]_{}&0\\
 0\ar[r]&\ds_{i\in I}(N\oo_{B}Y_{i})\ar[r]^{}&\ds_{i\in I}X_{i}\ar[r]&\ds_{i\in I}(X_{i}/\im(g_{i}))\ar[r]&0.}$$
 Thus $$\ds_{i\in I}Y_{i}/\im(\ds_{i\in I}f_{i})\cong \ds_{i\in I}(Y_{i}/\im(f_{i}))\in \mathcal{D}_{1}$$ and $$\ds_{i\in I}X_{i}/\im(\ds_{i\in I}g_{i})\cong \ds_{i\in I}(X_{i}/\im(g_{i}))\in \mathcal{C}_{1}$$ as $\mathcal{C}_{1}$ and $D_{1}$ are closed under direct sums.
 Thus $\ds_{i\in I}(X_{i},Y_{i},f_{i},g_{i})\in\mathfrak{B}^{\mathcal{C}_{1}}_{\mathcal{D}_{1}} $.

 Let $$0\ra (X_{1},Y_{1},f_{1},g_{1})\ra (X_{2},Y_{2},f_{2},g_{2})\ra (X_{3},Y_{3},f_{3},g_{3})\ra 0$$ be an exact sequence of left $\Delta$-modules
 with $(X_{1},Y_{1},f_{1},g_{1})\in \mathfrak{B}^{\mathcal{C}_{1}}_{\mathcal{D}_{1}}$ and $(X_{3},Y_{3},f_{3},g_{3})\in \mathfrak{B}^{\mathcal{C}_{1}}_{\mathcal{D}_{1}}.$
 Consider the following commutative diagrams with exact rows
 $$\xymatrix{
 &M\oo_{B}X_{1}\ar[r]^{}\ar[d]_{f_{1}}&M\oo_{A}X_{2}\ar[r]^{}\ar[d]_{f_{2}}&M\oo_{A}X_{3}\ar[r]\ar[d]_{f_{3}}&0\\
 0\ar[r]&Y_{1}\ar[r]^{}&Y_{2}\ar[r]&Y_{3}\ar[r]&0\\
 }$$
 and
$$\xymatrix{
 &N\oo_{B}Y_{1}\ar[r]^{}\ar[d]_{g_{1}}&N\oo_{B}Y_{2}\ar[r]^{}\ar[d]_{g_{2}}&N\oo_{B}Y_{3}\ar[r]\ar[d]_{g_{3}}&0\\
 0\ar[r]&X_{1}\ar[r]^{}&X_{2}\ar[r]&X_{3}\ar[r]&0\\
 }$$
Since $(X_{1},Y_{1},f_{1},g_{1})$ and $(X_{3},Y_{3},f_{3},g_{3})$ are in $ \mathfrak{B}^{\mathcal{C}_{1}}_{\mathcal{D}_{1}},$ $f_{1},g_{1},f_{3}$ and $g_{3}$ are monomorphisms.
By the Snake Lemma, we get that $f_{2}$ and $g_{2}$ are monomorphisms,
 $$0\ra Y_{1}/\im(f_{1})\ra Y_{2}/\im(f_{2})\ra Y_{3}/\im(f_{3})\ra 0$$
 and
  $$0\ra X_{1}/\im(g_{1})\ra X_{2}/\im(g_{2})\ra X_{3}/\im(g_{3})\ra 0$$
are exact.
Since $Y_{1}/\im(f_{1})$ and $Y_{3}/\im(f_{3})$ are in $\mathcal{D}_{1}$, $Y_{2}/\im(f_{2})\in \mathcal{D}_{1}$ as $\mathcal{D}_{1}$ is closed under extensions.
Similarly, we have $X_{2}/\im(g_{2})\in  \mathcal{C}_{1}.$
Thus $(X_{2},Y_{2},f_{2},g_{2})\in\mathfrak{B}^{\mathcal{C}_{1}}_{\mathcal{D}_{1}}.$
Hence $(\mathfrak{B}^{\mathcal{C}_{1}}_{\mathcal{D}_{1}},
\mathfrak{J}_{{\mathcal{C}_{2}},{\mathcal{D}_{2}}})$ is a perfect duality pairs.

(2). It is easy to see that $\mathcal{C}_{1}$ and $\mathcal{D}_{1}$ are closed under extensions and direct sums if and only if $\mathfrak{A}^{\mathcal{C}_{1}}_{\mathcal{D}_{1}}$ is closed under extensions and direct sums.

$ ``\Rightarrow".$ Since $_{\Delta}\Delta=\mathrm{T}_{A}(A)\oplus\mathrm{T}_{B}(B)\in \mathfrak{A}^{\mathcal{C}_{1}}_{\mathcal{D}_{1}},$ we have $A\oplus (N\oo_{B}B)\in \mathcal{C}_{1}$ and $B\oplus (M\oo_{A}A)\in \mathcal{D}_{1}$.
So $A\in \mathcal{C}_{1}$, $_{A}N\in \mathcal{C}_{1}$, $B\in \mathcal{D}_{1}$ and $_{B}M\in \mathcal{D}_{1}$.
Hence $(\mathcal{C}_{1},\mathcal{C}_{2})$ and $(\mathcal{D}_{1},\mathcal{D}_{2})$ are perfect duality pairs.

$``\Leftarrow".$ Since $(\mathcal{C}_{1},\mathcal{C}_{2})$ and $(\mathcal{D}_{1},\mathcal{D}_{2})$ are perfect duality pairs, we have $A\in \mathcal{C}_{1}$ and $B\in \mathcal{D}_{1}.$
Note that ${_{B}M}\in \mathcal{D}_{1}$ and ${_{A}N}\in \mathcal{C}_{1}$.
It follows that $\mathrm{T}_{A}(A)$ and $\mathrm{T}_{B}(B)$ are in
$\mathfrak{A}^{\mathcal{C}_{1}}_{\mathcal{D}_{1}}.$
Hence $_{\Delta}\Delta=\mathrm{T}_{A}(A)\oplus\mathrm{T}_{B}(B)$ is in
$\mathfrak{A}^{\mathcal{C}_{1}}_{\mathcal{D}_{1}}$ as $\mathcal{C}_{1}$ and $\mathcal{D}_{1}$ are closed under finite direct sums by Remark \ref{Remark1}.
So $(\mathfrak{A}^{\mathcal{C}_{1}}_{\mathcal{D}_{1}},
\mathfrak{A}_{{\mathcal{C}_{2}},{\mathcal{D}_{2}}})$ is a perfect duality pair.
\end{proof}

Recall from \cite{Gillespie} that a duality pair $(\mathcal{X},\mathcal{Y})$ of $R$-modules is complete if $(\mathcal{X},\mathcal{Y})$ is a symmetric duality pair with $(\mathcal{X},\mathcal{Y})$ being a perfect duality pair.

\begin{corollary}\label{complete} Let $\mathcal{C}_{1}$ (resp., $\mathcal{C}_{2}$) be a class of left (resp., right) $A$-modules and $\mathcal{D}_{1}$ (resp., $\mathcal{D}_{2}$) be a class of left (resp., right) $B$-modules.
\begin{enumerate}
\item If $N_{B}$ and $M_{A}$ are finitely generated projective modules, then $(\mathfrak{B}^{\mathcal{C}_{1}}_{\mathcal{D}_{1}},
\mathfrak{J}_{{\mathcal{C}_{2}},{\mathcal{D}_{2}}})$ is a complete duality pair if and only if $(\mathcal{C}_{1},\mathcal{C}_{2})$ and $(\mathcal{D}_{1},\mathcal{D}_{2})$ are complete duality pairs.

\item  $(\mathfrak{A}^{\mathcal{C}_{1}}_{\mathcal{D}_{1}},
\mathfrak{A}_{{\mathcal{C}_{2}},{\mathcal{D}_{2}}})$ is a complete duality pair if and only if ${ _{B}M}\in \mathcal{D}_{1}$ and ${_{A}N}\in \mathcal{C}_{1},$ $(\mathcal{C}_{1},\mathcal{C}_{2})$ and $(\mathcal{D}_{1},\mathcal{D}_{2})$ are complete duality pairs.
\end{enumerate}
\end{corollary}
\begin{proof}
It follows from Corollary \ref{symmetric} and Theorem \ref{perfect}.
\end{proof}

\section { \bf Gorenstein projective modules relative to duality pairs}\label{section:Gor-projective}

In this section, let $\Delta =\left(\begin{smallmatrix}  A & {_{A}}N_{B}\\  {_{B}}M_{A} & B \\\end{smallmatrix}\right)$ be a Morita ring such that $N_{B}$ and $M_{A}$ are finitely generated projective modules. We always assume that $(\mathcal{C}_{1},\mathcal{C}_{2})$ and $(\mathcal{D}_{1},\mathcal{D}_{2})$ are complete duality pairs over $A$ and $B$
respectively.
Then we have a complete duality pair $(\mathfrak{B}^{\mathcal{C}_{1}}_{\mathcal{D}_{1}},
\mathfrak{J}_{{\mathcal{C}_{2}},{\mathcal{D}_{2}}})$ over $\Delta =\left(\begin{smallmatrix}  A & {_{A}}N_{B}\\  {_{B}}M_{A} & B \\\end{smallmatrix}\right)$ by Corollary \ref{complete}.


\begin{definition}\cite[Definition 4.2]{Gillespie}\label{definition:4.1}
Let $\mathcal{L}$ (resp., $\mathcal{A}$) be a class of left (resp., right) $R$-modules over a ring $R$ such that $(\mathcal{L},\mathcal{A})$ is a complete duality pair of $R$-modules.
A left $R$-module $X$ is called Gorenstein $(\mathcal{L},\mathcal{A})$-projective if there exists an exact complex of projective left $R$-modules
$$\mathbb{P}:\cdots\ra P^{-1}\xrightarrow{d^{-1}} P^{0}\xrightarrow{d^{0}} P^{1}\ra\cdots$$
such that $\Hom_{R}(\mathbb{P}, L)$ is exact for any $L\in \mathcal{L}$ and that $X\cong \ker d^{0}.$
\end{definition}

The notion of Ding projective modules was first introduced by Ding, Li and Mao \cite{D} under the name of ``strongly Gorenstein flat modules''.
Later on, Gillespie \cite{Gillespie2} renamed these modules Ding projective modules.
For a ring $R$, a left $R$-module $M$ is called \emph{Ding projective} if there is an exact sequence of projective left $R$-modules
$$\mathbf{P}= \cdots\rightarrow P_{1}\rightarrow P_{0}\rightarrow P^{0}\rightarrow P^{1}\rightarrow\cdots$$
with $M\cong\mathrm{Ker}(P_{0}\rightarrow P^{0})$ such that $\mathrm{Hom}_{R}(\mathbf{P},F)$ is exact for any flat left $R$-module $F$.
For more details on Ding projective modules, we refer to \cite{D,Gillespie2,Liu,Mao1,Iacob}.
\begin{example}\label{ex1}
If $R$ is a right coherent ring, then $(\mathrm{Flat}(R),\mathrm{FP}\text{-}\mathrm{Inj}(R^{op}))$ is a complete duality pair by \cite[Theorem 1]{CS}, where $\mathrm{Flat}(R)$ denotes the class of flat left $R$-modules and
$\mathrm{FP}\text{-}\mathrm{Inj}(R^{op})$ denotes the class of $\mathrm{FP}$-injective right $R$-modules. Note that Gorenstein $(\mathrm{Flat}(R),\mathrm{FP}\text{-}\mathrm{Inj}(R^{op}))$-projective modules is exactly the Ding projective modules.
\end{example}

The following is the main result of this section which contains Theorem \ref{Gorenstein projective module'} in the introduction.
\begin{theorem}\label{Gorenstein projective module1}
Let $N_{B}$ and $M_{A}$ be finitely generated projective.

\begin{enumerate}
\item Assume that $N\oo_{B}D\in \mathcal{C}_{1}$ for any $D\in \mathcal{D}_{1}.$
If $X$ is a Gorenstein $(\mathcal{C}_{1},\mathcal{C}_{2})$-projective left $A$-module, then $\mathrm{T}_{A}(X)$ is a Gorenstein $(\mathfrak{B}^{\mathcal{C}_{1}}_{\mathcal{D}_{1}},
\mathfrak{J}_{{\mathcal{C}_{2}},{\mathcal{D}_{2}}})$-projective left $\Delta$-module.

\item Assume that $_{A}N$ and $_{B}M$ are projective modules, $M\oo_{A}C\in \mathcal{D}_{1}$ and $\mathrm{id}_{B}\Hom_{A}(N,C)<\infty$ for any $C\in \mathcal{C}_{1}$.
If $(X,Y,f,g)$ is a Gorenstein $(\mathfrak{B}^{\mathcal{C}_{1}}_{\mathcal{D}_{1}},
\mathfrak{J}_{{\mathcal{C}_{2}},{\mathcal{D}_{2}}})$-projective left $\Delta$-module, then $X$ is a Gorenstein $(\mathcal{C}_{1},\mathcal{C}_{2})$-projective left $A$-module.

\item Assume that $M\oo_{A}C\in \mathcal{D}_{1}$ for any $C\in \mathcal{C}_{1}.$
If $Y$ is a Gorenstein $(\mathcal{C}_{1},\mathcal{C}_{2})$-projective left $B$-module, then $\mathrm{T}_{B}(Y)$ is a Gorenstein $(\mathfrak{B}^{\mathcal{C}_{1}}_{\mathcal{D}_{1}},
\mathfrak{J}_{{\mathcal{C}_{2}},{\mathcal{D}_{2}}})$-projective left $\Delta$-module.

\item Assume that $_{A}N$ and $_{B}M$ are projective modules, $N\oo_{B}D\in \mathcal{C}_{1}$ and $\mathrm{id}_{A}\Hom_{B}(M,D)<\infty$ for any $D\in \mathcal{D}_{1}$.
If $(X,Y,f,g)$ is a Gorenstein $(\mathfrak{B}^{\mathcal{C}_{1}}_{\mathcal{D}_{1}},
\mathfrak{J}_{{\mathcal{C}_{2}},{\mathcal{D}_{2}}})$-projective left $\Delta$-module, then $Y$ is a Gorenstein $(\mathcal{D}_{1},\mathcal{D}_{2})$-projective left $B$-module.
\end{enumerate}
\end{theorem}

%
%
%
%
%
%
\begin{proof}
We only need prove (1) and (2), and the proofs of (3) and (4) are similar.

(1). Let $_{A}X$ be a Gorenstein $(\mathcal{C}_{1},\mathcal{C}_{2})$-projective module.
Then there exists an exact sequence of projective left $A$-modules
$$\mathbb{P}:\cdots\ra P^{-1}\xrightarrow{d^{-1}} P^{0}\xrightarrow{d^{0}} P^{1}\ra \cdots$$
such that $\Hom_{R}(\mathbb{P}, C)$ is exact for any $C\in \mathcal{C}_{1}$ and that $X\cong \ker d^{0}.$
Since $M_{A}$ is projective by hypothesis,
we have the exact sequence of projective $\Delta$-modules
$$\mathrm{T}_{A}(\mathbb{P}):\cdots\ra \mathrm{T}_{A}(P^{-1})\xrightarrow{\mathrm{T}_{A}(d^{-1})} \mathrm{T}_{A}(P^{0})\xrightarrow{\mathrm{T}_{A}(d^{0})} \mathrm{T}_{A}(P^{1})\ra\cdots$$
such that $\mathrm{T}_{A}(X)\cong \ker(\mathrm{T}_{A}(d^{0}))$ by Lemma \ref{lem1}.
For any $(X',Y',f',g')\in \mathfrak{B}^{\mathcal{C}_{1}}_{\mathcal{D}_{1}},$
we have exact sequences
$$0\ra N\oo_{B}Y'\xrightarrow{g'} X'\ra X'/\im(g')\ra 0$$
of $A$-modules with $X'/\im(g')\in \mathcal{C}_{1}$
and
$$0\ra M\oo_{A}X'\xrightarrow{f'} Y'\ra Y'/\im(f')\ra 0$$
of $B$-modules with $Y'/\im(f')\in \mathcal{D}_{1}.$
Note that $N\oo_{B}M=0$.
It follows that $N\oo_{B}Y'\cong N\oo_{B}Y'/\im(f).$
So $N\oo_{B}Y'\in \mathcal{C}_{1}$ by hypothesis.
Since $\mathbb{P}$ is an exact complex of projective $A$-modules,
we get an exact sequence of complexes
$$0\ra \Hom_{A}(\mathbb{P},N\oo_{B}Y')\ra \Hom_{A}(\mathbb{P},X')\ra \Hom_{A}(\mathbb{P},X'/\im(g'))\ra 0.$$
Since $X$ is a Gorenstein $(\mathcal{C}_{1},\mathcal{C}_{2})$-projective module, it follows that $\Hom_{A}(\mathbb{P},N\oo_{B}Y')$ and $\Hom_{A}(\mathbb{P},X'/\im(g'))$
are exact.
Hence $\Hom_{A}(\mathbb{P},X')$ is exact.
Note that
\begin{align*}
 \Hom_{\Delta}(T_{A}(\mathbb{P}),(X',Y',f',g'))\cong\Hom_{A}(\mathbb{P},X')
 \end{align*} by Lemma \ref{lem:adjoint pairs}.
So $\mathrm{T}_{A}(X)$ is a Gorenstein $(\mathfrak{B}^{\mathcal{C}_{1}}_{\mathcal{D}_{1}},
\mathfrak{J}^{\mathcal{C}_{2}}_{\mathcal{D}_{2}})$-projective left $\Delta$-module, as desired.

(2). Let $(X,Y,f,g)$ be a Gorenstein $(\mathfrak{B}^{\mathcal{C}_{1}}_{\mathcal{D}_{1}},
\mathfrak{J}_{{\mathcal{C}_{2}},{\mathcal{D}_{2}}})$-projective left $\Delta$-module.
Then we have an exact sequence of projective left $\Delta$-modules
$$\mathbb{T}:\cdots\xrightarrow{}\mathrm{T}_{A}(P^{-1})\oplus\mathrm{T}_{B}(Q^{-1})\xrightarrow{d^{-1}}
\mathrm{T}_{A}(P^{0})\oplus\mathrm{T}_{B}(Q^{0})\xrightarrow{d^{0}}\mathrm{T}_{A}(P^{1})\oplus\mathrm{T}_{B}(Q^{1})\ra\cdots$$
such that $\Hom_{\Delta}(\mathbb{T},(X',Y',f',g'))$ is exact for any $(X',Y',f',g')\in \mathfrak{B}^{\mathcal{C}_{1}}_{\mathcal{D}_{1}}$ and $\ke d^{0}\cong (X,Y,f,g),$
where each $P^{i}$ is a projective left $A$-module and each $Q^{i}$ is a projective left $B$-module.
Then we have the following exact sequence
$$\mathrm{U}_{A}(\mathbb{T}):\cdots\ra P^{-1}\oplus (N\oo_{B}Q^{-1})\xrightarrow{\mathrm{U}_{A}(d^{-1})}P^{0}\oplus (N\oo_{B}Q^{0})\xrightarrow{\mathrm{U}_{A}(d^{0})} P^{1}\oplus (N\oo_{B}Q^{1})\ra\cdots$$
such that $\ke(\mathrm{U}_{A}(d^{0}))\cong X.$
Since $_{A}N$ is projective by hypothesis, each $N\oo_{B}Q^{i}$ is a projective left $A$-module.
Thus each $P^{i}\oplus (N\oo_{B}Q^{i})$ is a projective left $A$-module.
For any $C\in\mathcal{C}_{1},$
consider the following exact sequence of left $\Delta$-modules
$$0\ra (0,M\oo_{A}C,0,0)\ra (C,M\oo_{A}C,1,0)\ra (C,0,0,0)\ra 0.$$
Since $\mathbb{T}$ is an exact complex of projective left $\Delta$-modules,
we get an exact sequence of complexes
\begin{align*}
\scalebox{1}{\xymatrixcolsep{3pc}\xymatrix{
0\ra\Hom_{\Delta}(\mathbb{T},(0,M\oo_{A}C,0,0))\ra
\Hom_{\Delta}(\mathbb{T},(C,M\oo_{A}C,1,0))\ra
\Hom_{\Delta}(\mathbb{T},(C,0,0,0))\ra 0.}}
\end{align*}
Note that $M\oo_{A}C\in \mathcal{D}_{1}$ and $N\oo_{B}M=0$ by hypothesis, it follows that
$(0,M\oo_{A}C,0,0)$ and $(C,M\oo_{A}C,1,0)$ are in $\mathfrak{B}^{\mathcal{C}_{1}}_{\mathcal{D}_{1}}.$
Thus $\Hom_{\Delta}(\mathbb{T},(0,M\oo_{A}C,0,0))$ and $\Hom_{\Delta}(\mathbb{T},(C,M\oo_{A}C,1,0))$ are exact.
So $\Hom_{\Delta}(\mathbb{T},(C,0,0,0))$ is exact.
Since $\mathrm{id}_{B}\Hom_{A}(N,C)<\infty$ by hypothesis,
there is an exact sequence of left $B$-modules
$$0\ra \Hom_{A}(N,C)\ra J^{0}\ra J^{1}\ra \cdots \ra J^{n-1}\ra J^{n}\ra 0$$
for some nonnegative integer $n$ with each $J^{i}$ injective.
Then we have an exact sequence of left $\Delta$-modules
$$0\ra \mathrm{H}_{B}(\Hom_{A}(N,C))\ra \mathrm{H}_{B}(J^{0})\ra \mathrm{H}_{B}(J^{1})\ra \cdots \ra \mathrm{H}_{B}(J^{n-1})\ra \mathrm{H}_{B}(J^{n})\ra 0$$
as $_{B}M$ is projective.
Note from the adjunction isomorphism that $\Hom_{B}(M,\Hom_{A}(N,C))\cong
\Hom_{A}(N\oo_{B}M,C)$. Then from $N\oo_{B}M=0$ we have $\Hom_{B}(M,\Hom_{A}(N,C))=0.$
Thus $\mathrm{H}_{B}(\Hom_{A}(N,C))=(0,\Hom_{A}(N,C),0,0).$
Note from Lemma \ref{lem1}(2) that each $\mathrm{H}_{B}(J^{i})$ is an injective left $\Delta$-module.
This implies that $\Hom_{\Delta}(\mathbb{T},\mathrm{H}_{B}(J^{i}))$ is exact for $0\leq i\leq n$.
It follows that $\Hom_{\Delta}(\mathbb{T},(0,\Hom_{A}(N,C),0,0))$ is exact by induction.
One sees that the following sequence of left $\Delta$-modules
$$0\ra (C,0,0,0)\ra (C,\Hom_{A}(N,C),0,e_{C})\ra (0,\Hom_{A}(N,C),0,0)\ra 0$$
is exact in $\DMod.$
Thus we have the following exact sequence of complexes
\begin{multline*}
0\longrightarrow \Hom_{\Delta}(\mathbb{T},(C,0,0,0))\longrightarrow \Hom_{\Delta}(\mathbb{T},(C,\Hom_{A}(N,C),0,e_{C}))\longrightarrow\\
\Hom_{\Delta}(\mathbb{T},(0,\Hom_{A}(N,C),0,0))\longrightarrow 0
\end{multline*}
as $\mathbb{T}$ is an exact complex of projective left $\Delta$-modules.
 Note that $\Hom_{\Delta}(\mathbb{T},(C,0,0,0))$ and $\Hom_{\Delta}(\mathbb{T},(0,\Hom_{A}(N,C),0,0))$ are exact in $\DMod$.
So $\Hom_{\Delta}(\mathbb{T},(C,\Hom_{A}(N,C),0,e_{C}))$ is exact in $\DMod.$
By Lemma \ref{lem:adjoint pairs}, we have
 \begin{align*}
 \Hom_{A}(\mathrm{U}_{A}(\mathbb{T}),C)&\cong \Hom_{\Delta}(\mathbb{T},\mathrm{H}_{A}(C)) =\Hom_{\Delta}(\mathbb{T},(C,\Hom_{A}(N,C),0,e_{C})).
\end{align*}
Thus $\Hom_{A}(\mathrm{U}_{A}(\mathbb{T}),C)$ is exact in $\DMod.$
So $_{A}X$ is a Gorenstein $(\mathcal{C}_{1},\mathcal{C}_{2})$-projective left $A$-module. This completes the proof.
\end{proof}

If the bimodule $_{A}N_{B}=0$, then  from Theorem \ref{Gorenstein projective module1} we have the following result.
\begin{corollary}
Let $T =\left(\begin{smallmatrix}  A & 0\\  {_{B}}M_{A} & B \\\end{smallmatrix}\right)$ be a triangular matrix ring with $M_{A}$ finitely generated projective.
\begin{enumerate}
\item If $_{A}X$ is a Gorenstein $(\mathcal{C}_{1},\mathcal{C}_{2})$-projective left $A$-module, then $(X,M\oo_{A}X,1,0)$ is a Gorenstein $(\mathfrak{B}^{\mathcal{C}_{1}}_{\mathcal{D}_{1}},
\mathfrak{J}_{{\mathcal{C}_{2}},{\mathcal{D}_{2}}})$-projective left $T$-module.

\item Assume that $_{B}M$ is projective, $M\oo_{A}C\in \mathcal{D}_{1}$ for any $C\in \mathcal{C}_{1}$.
If $(X,Y,f,0)$ is a Gorenstein $(\mathfrak{B}^{\mathcal{C}_{1}}_{\mathcal{D}_{1}},
\mathfrak{J}_{{\mathcal{C}_{2}},{\mathcal{D}_{2}}})$-projective left $T$-module, then $_{A}X$ is a Gorenstein $(\mathcal{C}_{1},\mathcal{C}_{2})$-projective left $A$-module.

\item Assume that $M\oo_{A}C\in \mathcal{D}_{1}$ for any $C\in \mathcal{C}_{1}.$
If $_{B}Y$ is a Gorenstein $(\mathcal{C}_{1},\mathcal{C}_{2})$-projective left $B$-module, then $(0,Y,0,0)$ is a Gorenstein $(\mathfrak{B}^{\mathcal{C}_{1}}_{\mathcal{D}_{1}},
\mathfrak{J}_{{\mathcal{C}_{2}},{\mathcal{D}_{2}}})$-projective left $T$-module.

\item Assume that $_{B}M$ is projective and $\mathrm{id}_{A}\Hom_{B}(M,D)<\infty$ for any $D\in \mathcal{D}_{1}$.
If $(X,Y,f,0)$ is a Gorenstein $(\mathfrak{B}^{\mathcal{C}_{1}}_{\mathcal{D}_{1}},
\mathfrak{J}_{{\mathcal{C}_{2}},{\mathcal{D}_{2}}})$-projective left $T$-module, then $_{B}Y$ is a Gorenstein $(\mathcal{D}_{1},\mathcal{D}_{2})$-projective left $B$-module.
\end{enumerate}
\end{corollary}

\begin{lemma}\label{flat module}
  A left $\Delta$-module $(X,Y,f,g)$ is flat if and only if $f$ and $g$ are monomorphisms, $X/\im(g)$ and $Y/\im(f)$ are flat.
\end{lemma}
\begin{proof}
Let $(X,Y,f,g)$ be a left $\Delta$-module.
Then we have the following exact sequences
$$M\oo_{A}X\xrightarrow{f}Y\ra Y/\im(f)\ra 0$$
of $B$-modules and
$$N\oo_{B}Y\xrightarrow{g}X\ra X/\im(g)\ra 0$$
of $A$-modules.
Since $M\oo_{A}N=0=N\oo_{B}M$, we have $N\oo_{B}Y\cong N\oo_{B}Y/\im(f)$ and
$M\oo_{B}X\cong M\oo_{B}X/\im(g).$

Assume that $(X,Y,f,g)$ is a flat left $\Delta$-module, it follows from \cite[Theorem 3.6.5]{KT} that $X/\im(g)$ and $Y/\im(f)$ are flat modules, $M\oo_{A}X/\im(g)\cong \im(f)$ and $N\oo_{B}Y/\im(f)\cong \im(g).$
Thus $M\oo_{B}X\cong \im(f)$ and $N\oo_{B}Y\cong \im(g).$
So $f$ and $g$ are monomorphism.

Conversely, if $f$ and $g$ are monomorphisms, $X/\im(g)$ and $Y/\im(f)$ are flat modules, then $M\oo_{B}X\cong \im(f)$ and $N\oo_{B}Y\cong \im(g).$
So $N\oo_{B}Y/\im(f)\cong \im(g)$ and $M\oo_{B}X/\im(g)\cong \im(f).$
Thus $(X,Y,f,g)$ is a flat left $\Delta$-module by \cite[Theorem 3.6.5]{KT}.
\end{proof}

\begin{lemma}\label{duality pair}
Let $(\mathcal{L},\mathcal{A})$ and $(\mathcal{L},\mathcal{A}')$ be complete duality pairs over a ring $R$.
Then $\mathcal{A}=\mathcal{A}'$.
\end{lemma}
\begin{proof}
Let $X\in \mathcal{A}$.
Then $X^{+}\in \mathcal{L}$. If $X$ is not in $\mathcal{A}'$, then $X^{+}$ is not in $\mathcal{L}$ as $(\mathcal{L},\mathcal{A}')$ is a complete duality pairs, which is a contradiction. So $X\in \mathcal{A}'$, that is, $\mathcal{A}\subseteq \mathcal{A}'$.
On the other hand, we can show that $\mathcal{A}'\subseteq \mathcal{A}$.
Hence $\mathcal{A}=\mathcal{A}'$.
\end{proof}

\begin{proposition}\label{prop:4.7}
 Let $\Delta$ be a right coherent ring with $N_{B}$ and $M_{A}$ finitely generated projective.
Then a right $\Delta$-module $(X,Y,f,g)$ is an $FP$-injective right $\Delta$-module if and only if $\tilde{g}:Y\ra \Hom_{A}(M,X)$ is a $B$-epimorphism, $\tilde{f}:X\ra \Hom_{B}(N,Y)$ is an $A$-epimorphism, $\ker \tilde{f}$ is an $FP$-injective right $B$-module and $\ker \tilde{g}$ is an $FP$-injective right $A$-module.
\end{proposition}
\begin{proof}
Since $\Delta$ is a right coherent ring, it follows from \cite[Theorem 5.4]{YY} that $A$ and $B$ are right coherent rings.
Thus this assertion follows from Corollary \ref{complete}, Example \ref{ex1},
Lemmas \ref{flat module} and \ref{duality pair}.
\end{proof}
\begin{corollary}\label{ding-projective}
 Let $\Delta$ be a right coherent ring with $N_{B}$ and $M_{A}$ finitely generated projective.
\begin{enumerate}
\item  Assume that $N\oo_{B}F$ is a flat left $A$-module for any flat left $B$-module $F.$
If $X$ is a Ding projective left $A$-module, then $\mathrm{T}_{A}(X)$ is a Ding projective left $\Delta$-module.

\item Assume that $M\oo_{A}Z$ is a flat left $B$-module for any flat left $A$-module $Z.$
If $Y$ is a Ding projective left $B$-module, then $\mathrm{T}_{B}(Y)$ is a Ding projective left $\Delta$-module.

\item Assume that $_{A}N,{_{B}M}$ are projective and  $\mathrm{id}_{B}\Hom_{A}(N,Z)<\infty$ for any flat left $A$-module $Z$.
If $(X,Y,f,g)$ is a Ding projective left $\Delta$-module, then $X$ is a Ding projective left $A$-module.

\item Assume that $_{A}N,{_{B}M}$ are projective and $\mathrm{id}_{A}\Hom_{B}(M,F)<\infty$ for any flat left $A$-module $F$.
If $(X,Y,f,g)$ is a Ding projective left $\Delta$-module, then $Y$ is a Ding projective left $B$-module.
\end{enumerate}
\end{corollary}
\begin{proof}
The corollary follows from Theorem \ref{Gorenstein projective module1},
Example \ref{ex1}, Lemma \ref{flat module} and Proposition \ref{prop:4.7}.
\end{proof}

\bigskip

\noindent\textbf{Yajun Ma}\\
Department of Mathematics, Nanjing University, Nanjing 210093, China.\\
E-mails: \textsf{13919042158@163.com}\\[1mm]
\textbf{Jiafeng L${\rm \ddot{u}}$}\\
 Department of Mathematics, Zhejiang Normal University,
 Jinhua 321004, China.\\
 jiafenglv@zjnu.edu.cn\\[1mm]
\textbf{Huanhuan Li}\\
School of Mathematics and Statistics, Xidian University,
 Xi'an 710071, China.\\
E-mail: \textsf{lihh@xidian.edu.cn}\\[1mm]
\textbf{Jiangsheng Hu}\\
School of Mathematics and Physics, Jiangsu University of Technology,
 Changzhou 213001, China.\\
E-mail: \textsf{jiangshenghu@hotmail.com}\\[1mm]
\end{document}